\documentclass[12pt]{article}
\title{Matrix Dufresne identities}
\author{B. Rider\footnote{Department of Mathematics, Temple University, Brian.Rider@temple.edu}  \and  B. Valk\'o\footnote{Department of Mathematics, University of Wisconsin - Madison, valko@math.wisc.edu}}
\usepackage{amsfonts,color}
\usepackage{graphicx}
\usepackage{amsmath}
\usepackage{amsthm}
\usepackage{amssymb}
\usepackage{mathrsfs}
\usepackage{euscript}

\RequirePackage{color}
\RequirePackage[colorlinks,urlcolor=my-blue,linkcolor=my-red,citecolor=my-green]{hyperref}
\definecolor{my-blue}{rgb}{0.0,0.0,0.6}
\definecolor{my-red}{rgb}{0.5,0.0,0.0}
\definecolor{my-green}{rgb}{0.0,0.5,0.0}

    \newtheorem{theorem}{Theorem}
    \newtheorem{lemma}[theorem]{Lemma}
    \newtheorem{proposition}[theorem]{Proposition}
    \newtheorem{corollary}[theorem]{Corollary}

\theoremstyle{definition} % For roman text in the body
    
    \newtheorem{remark}[theorem]{Remark}

\newcommand{\R}{{\mathbb R}}
\newcommand{\CC}{{\mathbb C}}

\newcommand{\E}{{\mathbb E}}
\newcommand{\ev}{{\rm   E}}
\newcommand{\pr}{\mbox{\rm P}}
\newcommand{\lstar}{{\raise-0.15ex\hbox{$\scriptstyle \ast$}}}

\newcommand{\vect}{\text{vect}}
\theoremstyle{remark} % For an italic header, more subtle than definition style

\newcommand{\ind}{{\bf 1}}
\newcommand{\cB}{\mathcal{B}}

\newcommand{\Tr}{{\rm{tr} }}
\def\eqd{\stackrel{\mathrm{(law)}}{=}}

\begin{document}

\maketitle

\begin{abstract}
We prove a version of the classical Dufresne identity for matrix processes. More specifically, we show that  the 
inverse Wishart laws on the space of positive definite $r \times r$ matrices can be realized by 
$ \int_0^{\infty} M_s^{} M_s^{T} ds $ in which
$t \mapsto M_t$ is a drifted Brownian motion on $GL_r(\mathbb{R})$. 
This solves a  problem  in the study of spiked random matrix ensembles which served as the original motivation
for this result.
Various known extensions of the Dufresne identity (and their applications) are also shown to have analogs in 
this setting. In particular, we identify matrix valued diffusions built from $M_t$ which generalize in a natural way the scalar processes
figuring  into the 
geometric  L\'evy and Pitman theorems of Matsumoto and Yor.

\end{abstract}

%\tableofcontents

\section{Introduction}

For $t \mapsto b_t$ a standard Brownian motion denote the associated geometric Brownian motion with drift, along with its (square) running integral by 
\begin{equation}
\label{expBM}
m_t =m_t^{(\mu)}=e^{ b_t + \mu t}, \qquad a_t^{(\mu)} = \int_0^t m_s^2 ds.
\end{equation}
We will use the convention that $\mu > 0$, with the choice of sign in the superscript of $a_t^{(\pm \mu)}$ reserved to produce an integral of $ (m_t)^2 = (m_t^{(\pm \mu)})^2$ either converging or diverging (almost surely) as $t \rightarrow \infty$. In certain situations (if it does not cause confusion), we will not denote the dependence on $\mu$ explicitly. 

The functional $a_t^{(\mu)}$  arises in a number of contexts including mathematical finance, diffusions in random environment,  Brownian motion on hyperbolic spaces, and continuum models of $1+1$ dimensional 
polymers (see \cite{MYReview} and references therein). Connected to the valuation of a certain perpetuity, Dufresne \cite{Duf1} established the fundamental identity in law,
\begin{equation}
\label{Dufresne1}
    a_{\infty}^{(-\mu)}  \eqd \frac{1}{2  \xi},
\end{equation}
in which $\xi$ has the $\mathrm{Gamma}(\mu)$ distribution, with density function $\frac{1}{\Gamma(\mu)} x^{\mu-1} e^{-x}$ on the positive half line. 

The Dufresne result provides one possible starting point to what is a vast collection of beautiful distributional identities for integrated geometric Brownian motion, much of which 
was pioneered by the work of  Matsumoto and Yor.
For instance, the following process level version of \eqref{Dufresne1} was proved in
  \cite{MY4}:
\begin{equation}
\label{Dufresne2} 
 \left\{ \frac{1}{a_t^{(-\mu)}} , \, t > 0 \right\} \eqd  \left\{ \frac{1}{a_t^{(\mu)}} +  \frac{1}{\tilde{a}_{\infty}^{(-\mu)}} , \, t > 0 \right\}.
\end{equation}
Here $\tilde{a}_{\infty}^{(-\mu)}$ is a copy of $a_{\infty}^{(-\mu)}$, independent of the original Brownian motion $b_t$. It is important to note that Dufresne himself had earlier established \eqref{Dufresne2} at fixed times in \cite{Duf2}. Further afield  extensions include ``geometric" versions of L\'evy's $M-X$ and Pitman's $2M-X$ theorem \cite{MY3}, a Brownian Burke property \cite{OconYor}, and the integrability of the O'Connell-Yor polymer model \cite{Oconnell1}.

Motivated by a problem in random matrix theory one of the authors and J.~Ram\'irez  were led to a conjectured Dufresne type identity for 
matrix processes \cite{RR2}. Here we prove that conjecture, and begin a program to extend the various results connected to the Dufresne identity to these  matrix processes. In Section \ref{sub:IntroDufresne} we state our matrix analogs of \eqref{Dufresne1} and \eqref{Dufresne2}. In Section \ref{sub:IntroXandZ} we introduce matrix diffusions which provide a possible generalization of those appearing in the just alluded to geometric L\'evy and Pitman theorems, and discuss their asymptotics and  intertwining properties.
Section \ref{sub:IntroBurke} states a partial Burke-type property for our matrix process. Finally in Section \ref{sub:IntroSpike}
we go back and describe the motivating spiked random matrix connection, and Section \ref{sub:IntroOpen} discusses some open problems and related results in the literature.

\subsection{Dufresne for matrix processes}
\label{sub:IntroDufresne}

The natural matrix extension $M_t=M_t^{(\mu)}$ of the geometric Brownian motion which arises in \cite{RR2} is defined by the $r \times r$ matrix 
It\^o equation
\begin{align} 
\label{Mt}
d M_t = M_t d B_t +(\tfrac12 + \mu)  M_t dt , \qquad  M_0=I, \quad t\ge 0
\end{align}
where $t \mapsto B_t$ is  the matrix valued Brownian motion comprised of independent standard Brownian motions 
$\{ b_{ij}(t) \}_{1 \le i,j \le r}$. 
Certainly this coincides with $m_t$ when $r=1$. 
%For $\mu=0$ the process  $G_t=M_t^{(0)}$ satisfies the Stratonovich equation $d G_t  = G_t \circ d B_t$ and for general $\mu$ we have  $M_t^{(\mu)} = e^{\mu t} G_t $. 
Note  that $M_t$ is rotational invariant: if $O$ is a fixed orthogonal matrix then $OM_tO^T$ has the same law as a process as $M_t$. 

As we will point out below in Section \ref{sec:DufresneBasic},  $M_t$ is almost surely invertible for all time and  for any $s>0$, the process $t\to M_s^{-1} M_{t+s}, t\ge 0$ has the same law as $M_t, t\ge 0$ and is independent of $\{ M_r, \, 0\le r \le s \}$.  Using the independent multiplicative increment property it is easy to extend $M_t$ for all $t\in \R$.   Either version of the  process may be referred to as the Brownian motion (with drift $\mu$) on the general linear group $GL_r$.

%
%
%Applying It\^o's formula one gets that $\det M_t$ is a geometric Brownian motion, and thus 
%$M_t$ is almost surely invertible for all time. 
%
%By the linearity of the sde (\ref{Mt}) it is immediate that for any $s>0$, the process $t\to M_s^{-1} M_{t+s}, t\ge 0$ has the same law as $M_t, t\ge 0$ and it is independent of $\{ M_r, \, r \le s \}$. Using the independent multiplicative increment property it is easy to extend $M_t$ for all $t\in \R$. Let $\tilde M_t, t\in [0,1]$ be an independent copy of $M_t, t\in [0,1]$ and for $t\in [-1,0)$ define $M_t=\tilde M_1^{-1} \tilde M_{1+t}$
%In other words, $M_t$ is a left invariant (multiplicative) Brownian motion on the general linear group $GL_r$. 

%We record the following statement about the growth of the norm of $M_t$, which will be used throughout the paper. It's proof %will be postponed to XXX. 

%\begin{lemma}\label{normlemma0}
%Fix $\mu\in \R$ and consider the process $M_t=M_t^{(\mu)}$. Then for any matrix norm $\|\cdot \|$ we have the following a.s.~limit: 
%\begin{align}
%\lim_{t\to \infty} \frac{1}{t} \log \| M_t\|=\mu+\frac{r-1}{2}.
%\end{align}
%\end{lemma}

%
%
%
%
%

Along with $M_t$ we also define the additive functional $A_t^{(\mu)} = \int_0^t M_s^{} M_s^T ds $ which is the matrix analog of the running integral $a_t^{(\mu)}$ from (\ref{expBM}).  Our basic matrix Dufresne identity is the following.

\begin{theorem}
\label{thm:MatrixDufresne}
If   $2\mu>r-1$, the $r \times r$ random matrix
\begin{equation}
\label{MatrixDufresne}
\displaystyle A_{\infty}^{(-\mu)} =  \int_0^\infty  M_s  M_s^T ds
\end{equation}
has the standard inverse Wishart distribution with parameter $2\mu$. 
%Note that here $M_t=M_t^{(-\mu)}$. 
%The density is given by $g_\mu(x)=C \det(\xi)^{-\mu-\frac{r+1}{2}} e^{-\frac12 \Tr \xi^{-1}} $. 
\end{theorem}

As Lemma \ref{normlemma} below shows, $\lim_{t\to \infty} \frac{1}{t} \log \| M_t^{(-\mu)}\|= -\mu+\frac{r-1}{2}$ with  any matrix norm $\| \cdot \|$. The condition $2 \mu > r-1$ ensures that $A_\infty^{(-\mu)}$ is almost surely finite. That condition is also
necessary for the nondegeneracy of underlying Wishart distribution.

Recall that the standard  $r \times r$ (real) Wishart distribution with parameter $p > r-1$ is the law on the cone of symmetric positive definitive matrices 
${\mathcal{P}}$ prescribed by:\footnote{Since all matrix variables will reside in $GL_r$, we often omit the dependence on $r$ from the notation for the various distributions as well as their support, e.g., $\mathcal{P}$.}
\begin{equation}
\label{WishartLaw}
 \gamma_{p}(dX) = \frac{1}{\Gamma_r(p/2)} (\det X )^{   \frac{p-r-1}{2}} e^{-\frac12 \Tr X } \,  \ind_{\mathcal{P}}(X) dX.
 \end{equation}
 Here $\Gamma_r(p/2)$  the multivariate gamma function
 $
 \Gamma_r(p/2)=\pi^{\tfrac{r(r-1)}{4}} \prod_{k=1}^r \Gamma(\tfrac{p-k+1}{2}). 
 $
  When $p$ is also an integer $\gamma_{p}$ can be realized by the random sample covariance matrix $G G^T$ for $G$ an $ r \times p$ matrix with independent  standard normal entries. In either case it is the natural multivariate generalization of the gamma distribution.
 In symbols then Theorem \ref{thm:MatrixDufresne} reads $A_{\infty}^{(-\mu)} \eqd \gamma_{2 \mu}^{-1}$, the latter having density proportional to $   (\det X )^{ -\mu - \frac{r+1}{2}} e^{-\frac12 \Tr X^{-1} }$ on $\mathcal{P}$.  There are of course complex and quaternion Wishart distributions. Corollary \ref{cor:ComplexCase}  below provides a version of 
 Theorem \ref{thm:MatrixDufresne} for these settings.
 
The original  Dufresne identity \eqref{Dufresne1}  has a number of different proofs, not all of which appear extendable beyond the scalar case. To highlight those ideas that do carry over to the matrix case, we give two different proofs of Theorem \ref{thm:MatrixDufresne}. Both appear in Section \ref{sec:DufresneBasic}.
The first uses a inversion strategy employed by several authors.
The second mimics an argument of  Baudoin-O'Connell \cite{BaudOcon} which is  likely the most succinct proof of the one dimensional identity and which we briefly summarize now. 

Let  $y_t = y_0 e^{2b_t - 2 \mu t} $, which is to say that $y_t$ is  $m_t^2$ with a variable starting point and the convergent choice of the sign of $\mu$. The observation in \cite{BaudOcon} is that  
$$
  u(y) = \ev \left[ e^{-\frac{1}{2} \int_0^{\infty} y_t dt }  | y_0 = y \right] =     \ev  \left[ e^{-\frac{1}{2} y \int_0^{\infty} y_t dt }  | y_0 = 1 \right] = 
   \ev \left[ e^{-\frac{1}{2} y a_{\infty}^{(-\mu)}} \right], 
$$
is on one hand the Laplace transform of the desired distribution, and on the other, courtesy of Feyman-Kac, a solution of 
 \begin{equation}
\label{bessel1d}
    2 y  \frac{d}{dy} y  \frac{d}{dy} u(y) - 2\mu y \frac{d}{dx} u(y) - \frac{1}{2} y u(y) = 0, \qquad u(0) = 1.
\end{equation}
The unique bounded solution  of \eqref{bessel1d} is then shown to be $u(y) = \frac{2^{1-\mu}}{\Gamma(\mu)} y^{-\mu/2} K_{\mu} (\sqrt{y})$, where
\begin{equation}
 \label{Macdonald1d}
   K_{s}(a) = \frac{1}{2} \int_0^{\infty} x^{s-1} e^{-\frac12 a (x + \frac{1}{x} )} dx
\end{equation}
is the Macdonald function (or modified Bessel function of the second kind). After a change of variables $u$ is recognized as the Laplace transform of the (scaled) inverse gamma distribution.
 
 Bessel functions of a matrix argument first appear in the 1955 work of Herz \cite{Herz}, and we introduce what is effectively his ``$K$-Bessel" function: 
\begin{equation}
\label{KBessel1}
  K_r( s | A, B) =  \frac{1}{2} \int_{\mathcal{P}}  (\det X )^{s - \frac{r+1}{2} } e^{ - \frac12 {\Tr A} X - \frac12 {\Tr B} X^{-1} } \, dX,
 \end{equation}
 for $A,B \in P$.\footnote{Herz actually denotes what is effectively this function by $B_r$.  We follow more closely the notation of Terras  \cite[\S 4.2.2]{Terras}, where
 this is referred to as the $K$-Bessel function of the second kind. We choose a slightly different normalization  here (as in \cite{Letac}) by introducing the extra $1/2$ constants in the exponential term to better align with the standard $(r=1)$ Macdonald function.}
 Note this  reproduces the regular Macdonald function in the form 
 $K_1(s | a, b) = (ab)^{s/2} K_s(\sqrt{ab})$ upon setting  $r=1$ and $A,B = a, b \in \mathbb{R}_+$. Both functions are well defined for all $s \in \CC$. It is also clear that, up to a normalization, $K_{r} ( -\mu | A, I) $ is  the Laplace transform (in the variable $A$) of the $\gamma_{2\mu}^{-1}$ distribution.

Picking up on the basic idea in \cite{BaudOcon} we set $ Y_t =  M_t^{} M_t^T $ with $M_t = M_t^{(-\mu)}$ and $2 \mu> r-1$,
so that $A_{\infty}^{(-\mu)} = \int_0^{\infty} Y_t dt $.

\begin{theorem} 
\label{thm:MatrixBessel}
The process $t \mapsto Y_t  \in \mathcal{P}$ is Markovian with generator,
\begin{equation}
\label{YGenerator}
G_Y=2  {\Tr} (Y \frac{\partial}{\partial Y} )^2-2\mu {\Tr} (Y  \frac{\partial}{\partial Y} ), 
\end{equation}
 expressed here through the matrix-valued operator 
 $[\frac{\partial}{\partial Y} ]_{ij} = (\frac{1}{2} +  \frac{1}{2} \delta_{i,j} ) \frac{\partial}{\partial  Y_{ij}} $.
  
  Furthermore, for $2 \mu > r-1$
 the unique bounded  solution of
 \begin{equation}
 \label{KBessel2}
    G_Y U(Y) - \frac{1}{2} ( {\Tr }Y ) \, U(Y) = 0, \qquad U(0) = 1,\qquad Y\in \mathcal{P}
\end{equation}
is the normalized $K$-Bessel function $U(Y) = \frac{ K_{r} ( - \mu |  Y, I) }{2^{\mu-1} \Gamma_r(\mu)}$.
\end{theorem}

Theorem \ref{thm:MatrixDufresne} then follows from considerations similar to those above:
$$
 \ev[  e^{-\frac{1}{2} \Tr (Y A_{\infty}^{(-\mu)} ) } ] = \ev [  e^{-\frac{1}{2} \Tr (Y  \int_0^{\infty} Y_t dt )} | Y_0 = I ] =
  \ev [  e^{-\frac{1}{2} \Tr ( \int_0^{\infty} Y_t dt )} | Y_0 = Y  ] = U(Y).
$$
The middle equality uses that $Y_t$ started from $Y \in \mathcal{P}$ is equal in law to $\sqrt{Y} Y_t \sqrt{Y}^T$, with now $Y_t$ started from the identity, along with the trace being cyclic. 

The theory of matrix Bessel functions has been developed considerably since \cite{Herz}, in part due to applications to multivariate statistics as well as to the harmonic analysis of symmetric spaces. See for example \cite{Muirhead} (particularly Chapter 7) and  \cite{Terras}, respectively. Both references include a number of differential operator  characterizations of various matrix Bessel functions. Still, the present characterization of $K_r(\cdot | A, I)$  appears  
new  despite the obvious similarities of \eqref{bessel1d} and \eqref{KBessel2}. 

\begin{remark} The process $Y_t$ (modulo drift) was previously studied in \cite{Norris} as one of two canonical ``Brownian motion on ellipsoids".
Its Markov property, along with that of its joint process of eigenvalues, was already remarked upon there. Because of the rotational invariance of $M_t$, the function  $U(Y)$ is actually determined by the eigenvalues $ \Lambda = \Lambda(Y) = \lambda_1 \ge \lambda_2 \ge \cdots  \ge \lambda_r$ alone.  The eigenvalue process has generator   
\begin{equation*}
\label{GLambda1}
   G_{\Lambda} =  \sum_{k=1}^r  ( 2 \lambda_k  \frac{\partial}{\partial \lambda_k}  \lambda_k \frac{\partial}{\partial \lambda_k } - (r-1 +2\mu) \lambda_k  \frac{\partial}{\partial \lambda_k } ) 
   + \sum_{k < \ell}  \frac{1}{\lambda_k - \lambda_\ell}   (  2 \lambda_k^2 \frac{\partial}{\partial \lambda_k }
                        -  2 \lambda_{\ell}^2 \frac{\partial}{\partial \lambda_\ell } ),
\end{equation*} 
and thus \eqref{KBessel2} can be expressed instead by $(G_{\Lambda} - \frac{1}{2} \sum_{k=1}^r \lambda_k ) U(\Lambda) = 0$.
\end{remark}

Last, we also have the exact matrix analog of the process level Dufresne identity \eqref{Dufresne2}.

\begin{theorem} 
\label{thm:ProcessDufresne}
There is the following identity in distribution:
\begin{align}
\label{eq:ProcessDufresne}
 \left\{  (A^{(\mu)}_t)^{-1}, \,  t \ge 0 \right\} \eqd  \left\{  (A_t^{(-\mu)})^{-1}-  (A_\infty^{(-\mu)} )^{-1},  \, t\ge 0 \right\}.
\end{align}
Again it is assumed that $2\mu > r-1$.
%where $\tilde Q_t=\int_0^t \tilde M_s \tilde M_s^T ds$ with $d\tilde M=\tilde M d\tilde B+(\mu+1/2)  \tilde M dt$,  
%$\tilde M_0=I$.\\
%Moreover, we can make this an a.s.~identity with 
%\[
%\tilde B_t=B_t-2\mu t I+\int_0^t  M_s^T (Q_\infty-Q_s)^{-1} M_s ds. 
%\]
\end{theorem}

Theorem \ref{thm:ProcessDufresne} is proved in Section \ref{sec:ProcessDufresne} by an enlargement of filtration argument, similar to the  proof  in \cite{MY4} for the scalar case.
Note the slightly different, but equivalent,  presentation of the identity compared with \eqref{Dufresne2}. In this form \eqref{eq:ProcessDufresne} can actually be made an almost sure identity by an appropriate construction of the 
underlying matrix Brownian motions (see  Proposition \ref{claim:enlarge}).

\subsection{Geometric L\'evy and Pitman theorems}
\label{sub:IntroXandZ}

Connected to their study of the functional $a_t^{(\mu)}$, Matsumoto-Yor introduced the pair of processes,
\begin{equation}
\label{xandz} 
x_t =  m_t^{-2} \int_0^t m_s^2 ds, \qquad z_t = m_t^{-1} \int_0^t m_s^2 ds, 
\end{equation}
(with $m_t=m_t^{(\mu)}$), both of which turn out to be diffusions \cite{MY1, MY2, MY3}. For $x_t$ the Markov property is immediate. An application of It\^o's formula
produces the following simple sde for $x_t$ for any $\mu\in \R$:
\begin{equation}
\label{xdiffusion}
  dx_t = 2 x_t db_t +  dt + (2 - 2 \mu) x_t  dt.
\end{equation}
Plainly, the same procedure applied to  $z_t$ cannot produce a closed equation. Nonetheless,  $z_t$ is a Markov process (for any $\mu\in \R$)
with law described by,
\begin{equation}
\label{zdiffusion}
  d z_t =  z_td \bar{b}_t + ( \frac{1}{2} - \mu) z_t  dt + \frac{K_{\mu+1}}{K_{\mu}} \left( \frac{1}{z_t} \right) dt.
\end{equation}
Here $K_{\mu}$ is the Macdonald function \eqref{Macdonald1d} and  $\bar{b}_t$ is a  new Brownian motion (the subtlety is explained momentarily). 
An important property of $z_t$ is its invariance under the transformation $\mu \mapsto -\mu$ which follows from 
 identity 
$K_{\mu-1}(a) = K_{\mu+1}(a) - (2 \mu /a) K_{\mu}(a)$ along with the more transparent fact $K_{\mu}(a) = K_{-\mu}(a)$.

The interest in $x_t$ and $z_t$ is that they encode generalizations of the classical $M-X$ theorem of L\'evy, as well as the $2M-X$ theorem of Pitman, as was discovered by Matsumoto-Yor \cite{MY1}. In particular, rescaling time by $c^2$ and 
taking $\mu$ into $\gamma/c$ yields:
$$
  x_{c^2t}^{\gamma/c}  \eqd c^2 \int_0^t e^{c(2 b_s^{\gamma} - 2b_t^{\gamma})} ds, \qquad z_{c^2 t}^{\gamma/c} \eqd  c^2 \int_0^t e^{c( 2 b_s^{\gamma} - b_t^{\gamma})} ds,
$$
in which $b_t^{\gamma}$ is shorthand for the drifted Brownian motion.  Simple Laplace asymptotics yield that, as $c \rightarrow \infty$,  $\tfrac{1}{2c} \log  \int_0^t e^{c(2 b_s^{\gamma} - 2b_t^{\gamma})} ds $ and $ \tfrac1{c} \log  \int_0^t e^{c( 2 b_s^{\gamma} - b_t^{\gamma})} ds$  converge pathwise to $ \max_{s< t} (b_s^{\gamma} - b_t^{\gamma})$ and $\max_{s < t} (2 b_s^{\gamma} - b_t^{\gamma})$. Working on the sde 
side  (that is, with \eqref{xdiffusion} and \eqref{zdiffusion}), shows that $\tfrac1{2c} \log  x_{c^2t}^{\gamma/c} $ and $\tfrac1{c} \log z_{c^2 t}^{\gamma/c}$
have  limiting processes that are equivalent in law to the diffusions with respective generators
\begin{equation}
\label{xzGenerators}
  G_x =   \frac{1}{2} \frac{d^2}{dx^2} - \gamma \, \mathrm{sgn}(x)  \frac{d}{dx}, \qquad G_z  = \frac{1}{2} \frac{d^2}{dz^2} + \gamma \cot (\gamma z) \frac{d}{dz}. 
\end{equation}
The former is understood to be equipped with a Neumann boundary condition at the origin: the limiting $x$-process is reflected at the origin while the $z$-process has an entrance boundary at that point. 
 
Letting $\gamma \downarrow 0$,  from the processes  \eqref{xzGenerators} we recover the reflected Brownian motion and the 3-d Bessel process occurring in the celebrated results of 
 L\'evy and Pitman identifying the distributions of the processes $t\to \max_{s< t} (b_s - b_t)$ and $t\to \max_{s< t} (2b_s - b_t)$. 
  Taking the point of view that `$\exp \int \log $' has replaced the running maximum, that the Brownian functionals in  \eqref{xandz} are diffusions is now  referred to as  the ``geometric" $M-X$ or $2M-X$ theorem.

Similar to the original Dufresne identity \eqref{Dufresne1}, there are various ways to identify the law of $z_t$ 
with \eqref{zdiffusion}.  The one relevant here is again due to Matsumoto-Yor \cite{MY3},  and rests on properties of the Generalized Inverse Gaussian (GIG) distribution. The GIG is a three-parameter distribution on the positive half line 
with density proportional to $x^{p-1} e^{-\frac{1}{2} a x -\frac{1}{2} b x^{-1}}$, with arbitrary $p$ and $a, b >0$. As such it is intimately connected to the Macdonald  function $K_p$:  the ratio  appearing in the drift \eqref{zdiffusion} being  the mean of a GIG with parameters $(\mu, 1/z_t, 1/z_t)$.  What is proved in \cite{MY3} is that the law of $m_t$ conditional on the field $\{z_s, s \le t\}$ is exactly this GIG, and the closed equation \eqref{zdiffusion} is produced by a projection (and so the indicated $\bar{b}_t$ is measurable with respect to $ \sigma(z_s, s \le t)$).

By analogy with \eqref{xandz}  we introduce
\begin{equation}
\label{XandZ}
   X_t = M_t^{-1} (  \int_0^t M_s^{} M_s^T \, ds  )  M_t^{-T}, \qquad Z_t = M_t^{-1} \int_0^t  M_s^{} M_s^T \, ds,
\end{equation}
for our matrix process $t \mapsto M_t$ as defined originally in \eqref{Mt}. Note that from now on we use the shorthanded  $M^{-T}=[M^T]^{-1}$. 
That $X_t$ is Markovian is again straightforward using It\^o's formula. For $Z_t$,
as can be anticipated at this point, there is a matrix GIG distribution on $\mathcal{P}$, more or less defined by having normalizer given by the $K$-Bessel function \eqref{KBessel1}. That is, it is the law
\begin{equation}
\label{matrixGIG}
  \eta_{p, A, B}(dX) = c (\det X)^{ p - \frac{1}{2}(r+1)} e^{-\frac{1}{2}( \Tr  A X + \Tr B X^{-1} )} \ind_{\mathcal{P}}(X) dX,
\end{equation}
with $A, B \in P$ and $c =  \frac{1}{2} K_r(p | A,B)^{-1}$.

\begin{theorem}
\label{thm:MatrixXandZ}
%Both  $X_t$ and  $Z_t$ are diffusions. 
For all $\mu$, the process $X_t$ is the diffusion defined by the It\^o equation
\begin{equation}
\label{Xprocess}
   dX_t =  I dt - 2 \mu X_t dt + \Tr X_t  I dt - dB_t X_t - X_t^{} dB_t^{T},
\end{equation}
run on the same Brownian motion $B_t$ is as  $M_t$ \eqref{Mt}.

If $|\mu|>\frac{r-1}{2}$ the process $Z_t$ is also a diffusion. It satisfies 
\begin{equation}
\label{Zdiffusion}
dZ_t = d \bar{B}_t Z_t +(\frac12 - \mu) Z_t dt+ \kappa_{\mu}( I,  (Z_t {Z_t}^T)^{-1}) Z_t  dt
\end{equation}
where now $\bar{B}_t $ is a matrix valued Brownian motion adapted to $\sigma( Z_s, s\le t)$ and 
  $\kappa_p(A,B)$ denotes the mean of the $\eta_{p, A, B}$ distribution \eqref{matrixGIG}. 
In addition,
the law of $t \mapsto Z_t$ 
is unchanged by taking $\mu$ into $-\mu$.
\end{theorem}

While both $Z_t$ and the right hand side of \eqref{Zdiffusion} are sensible for all $\mu$, our method uses Theorem \ref{thm:MatrixDufresne} as input and so requires the  same condition. One assumes this gap might be filled by other means.

In general the mean of a matrix GIG does not appear to have a particularly nice expression.  Though if $A$ and $B$ are diagonal, one can see that $\kappa_{p}(A, B)$ is diagonal as well.  And by bringing in a (well known) generalization of the $K$-Bessel functions  introduced thus far one can get a reasonable handle on these diagonal components. See the proof of Theorem \ref{thm:Laplace} below for both points.  Note that the invariance of $Z_t$ under the map $\mu \mapsto -\mu$ implies the identity 
$ \kappa_{\mu}(I,A) = 2 \mu I + \kappa_{-\mu}(I,A)$, for $|\mu|>\tfrac{r-1}{2}$. (A standard analytic continuation argument  extends the identity to all $\mu\in \R$.)
A direct verification of  this identity seems laborious (and non-trivial).

 The key to \eqref{Zdiffusion} is that
the conditional distribution of $M_t$ given  $\{ Z_s, s \le t, Z_t = Z \}$ is $Z^T \Xi$ for 
$\Xi \sim \eta_{\mu, I,   (Z Z^T)^{-1} }$. This
hinges on  a characterization of the matrix GIG law due to Bernadac \cite{Bernadac}, which in turn builds on earlier work of
Letac-Wesolowski \cite{Letac}.  An immediate consequence of this is the following.

\begin{corollary} 
\label{cor:Intertwining} Let $|\mu|>\tfrac{r-1}{2}$. 
The laws of $M_t$ and $Z_t$ intertwine. Denote by $T_t^M$ and $T_t^Z$ the corresponding semigroups and define the Markov kernel $\Lambda$ as %from $GL_r$ to itself by,
\begin{equation}
\label{eq:Intertwining}
  \Lambda h( Z) =  \int_\mathcal{P} h(Z^T X  )  \, \eta_{\mu, I,  (ZZ^T)^{-1} } (dX),
\end{equation}
for all  suitable test functions $h: GL_r \mapsto \R$. Then it holds that ${\Lambda } T_t^M = T_t^Z \Lambda$.  Since $X_t = Z_t M_t^{-T}$, it follows that $X_t$ also intertwines with $Z_t$. In this case,  
\begin{equation}
\label{eq:Intertwining2}  
    \tilde{\Lambda} h(Z) =  \int_\mathcal{P} h(X^{-1})  \, \eta_{\mu, I, (ZZ^T)^{-1} } (dX),
\end{equation}
defines the corresponding kernel for which $\tilde{\Lambda } T_t^X = T_t^Z \tilde{\Lambda}$.
\end{corollary}

This intertwining has had far reaching implications in the scalar case.
A remaining question here  is whether the matrix processes contain either an $M-X$ or $2M-X$ type theorem.  We  show this occurs at the level  of the eigenvalues (or singular values) of $X_t$ and $Z_t$, each of which comprise their own Markov process.

\begin{theorem}
\label{thm:Laplace}
Denote by $x_r^{\mu}(t) \le \cdots \le x_1^{\mu}(t)$ the eigenvalues of $X_t$. Denote the ordered  
singular values of $Z_t$ similarly by $z_i^\mu(t)$. Speeding up time by a factor of $c^2$ and rescaling 
$\mu$ as in $\mu = \frac{r-1}{2} + \gamma/c$ for a fixed 
$\gamma >0$ we have that
$$
    \lim_{c \rightarrow \infty} \frac{1}{2c} \log x_1^{\gamma/c}(c^2 t) \Rightarrow  | b_t^{-\gamma \mathrm{sgn}(\cdot)} |, 
$$ 
and 
$$
    \lim_{c \rightarrow \infty} \frac{1}{c} \log z_r^{\gamma/c}(c^2 t) \Rightarrow  b_t^{\gamma \coth( \gamma \,\cdot) }.
$$
The notations indicate a reflected Brownian motion with constant drift $-\gamma$, and  a Brownian motion with variable drift  $\gamma \coth (\gamma \, \cdot)$,
respectively. In both cases the convergence takes place in the usual Skorohod topology.
\end{theorem}

In either case, the reminder of the spectrum has a relatively trivial limit in the chosen scaling. For $X_t$, each of  the similarly scaled lower eigenvalues  converge to the zero process. For $Z_t$, the larger singular values escape to infinity at increasing exponential rates. Note as well that Theorem \ref{thm:Laplace} provides  an analog for just half of the Matsumoto-Yor result $-$ one might like at the same time to have  path-wise identities by applying some sort of Laplace asymptotics to the definitions \eqref{XandZ}.

The proofs of Theorems \ref{thm:MatrixXandZ} and \ref{thm:Laplace}, along with that of Corollary \ref{cor:Intertwining}, are found in Section \ref{sec:XandZ}.

\subsection{Burke properties}
\label{sub:IntroBurke}

O'Connell-Yor \cite{OconYor} proved the following ``Brownian Burke property". Let $b_t$ and $c_t$ be independent Brownian motions, and 
set 
\begin{equation}
\label{TITO1}
    r_t = \log \int_{-\infty}^t e^{ b_{(s,t)} + c_{(s,t)} - \mu (t-s)} ds,
\end{equation}
where $b_{(s,t)} = b_t - b_s$  and $c_{(s,t)} = c_t - c_s$.
Then
\begin{equation}
\label{TITO2}
    \left. \begin{array}{l} b_t + r_0 - r_t , t\in \R\\   c_t + r_0 - r_t, t\in \R \end{array}  \right\} \eqd \left\{b_t, t\in \R\right\} \mbox{ and are independent}.
\end{equation}
%Further,  for fixed $t \in \R$,  the field generated by the pair of new Brownian motions over $s \le t$ is independent of $\{ r_s, s \ge t\}$.
 The analogy with the classical Burke property is made by considering $t \mapsto r_t$ as a generalized queue, where  `$\log \int \exp$' again replaces `$\sup$', with $t \mapsto b_t $ and $t \mapsto \mu t - c_t$ the respective arrival and departure processes.  This result is key in the construction of the semi-directed Brownian polymer (also introduced in \cite{OconYor}) which is  now understood to be a member of the KPZ universality class \cite{BorCorFer}. A similar Burke type property lies behind the integrability of Sepp\"al\"ainen's log-gamma polymer \cite{Sep} which has also subsequently been shown to have Tracy-Widom  fluctuations \cite{BorCorRem}.

The above scheme constructs two new independent Brownian motions from two independent input Brownian motions. As a preliminary step, a similar statement is established in \cite{OconYor} that shows that 
\begin{equation}
\label{OIOO}
\left\{   b_t^{\mu} + \alpha_t - \alpha_0, t \ge 0\right\} \eqd \left\{b_t^{\mu}, t \ge 0\right\} \quad \mbox{ where } \alpha_t = \log \int_{-\infty}^t e^{ 2 b_s^{\mu} - 2  b_t^{\mu}} ds.
\end{equation}
Here we are reusing notation from before, $b_t^{\mu}$ denoting a Brownian motion with drift $\mu$. Also, for fixed $\tau \in \R$,  the field generated by the new drifted Brownian motion
$ b_t^{\mu} + \alpha_t - \alpha_0$ over $t  \le \tau $ is independent of $\{ \alpha_t, t \ge \tau\}$. 

%While we do not have an appropriate version of a matrix process Burke property, that is,  of \eqref{TITO1} and \eqref{TITO2}, we do have 

The following provides a matrix extension of \eqref{OIOO}. It requires the full line version of the process $M_t$, the details of which are again described at the beginning of Section \ref{sec:DufresneBasic}. 

\begin{theorem} 
\label{thm:OIOO} Fix $2\mu>r-1$ and 
let $M_t=M_t^{(\mu)}$ be the solution of \eqref{Mt} extended over $t \in (-\infty, \infty)$. 
 Then, 
\begin{equation}
\label{eq:MatrixOIOO}
  \left( \int_{-\infty}^0 M_s^{} M_s^T ds \right)^{-1}   M_t \left(  \int_{-\infty}^t   M_t^{-1} M_s^{} M_s^T  M_t^{-T} ds    \right) \eqd M_t,   
\end{equation}
as processes for $t \ge 0$. In addition,
for any fixed $\tau >0$, the process defined by the left hand side of \eqref{eq:MatrixOIOO} up to time $\tau$ is independent
of $t \mapsto  M_t^{-1} (\int_{-\infty}^t M_s^{} M_s^T ds) M_t^{-T}$ for  $t \ge \tau$.
\end{theorem}

Note that
$ M_t^{-1} (\int_{-\infty}^t M_s^{} M_s^T ds) M_t^{-T}$ reduces to $e^{\alpha_t}$
 (and so \eqref{eq:MatrixOIOO} reduces precisely to \eqref{OIOO}) for $r=1$. 
 
 The next result is a matrix version of  \eqref{TITO2}, the Brownian Burke property of O'Connell-Yor.

\begin{theorem}\label{thm:TITO}
Let $B_t$ and $C_t$ be independent two-sided matrix Brownian motions and set $2 \mu > r-1$. Consider
the strong solution of 
\[
dH_t=H_t(dB_t+dC_t)+(2\mu+1) H_t dt, \qquad H_0=I,
\] 
extended to the whole line, again using the (multiplicative) independence stationary increment property as described in Section \ref{sec:DufresneBasic}. 
Now define the processes $A_{(-\infty,t)}=\int_{-\infty}^s H_u H_u^T du$ and 
\begin{align*}
F_t=B_t+2\mu I t- \frac12 \int_0^t H_s^TA_{(-\infty,s)}^{-1} H_s ds,\qquad
G_t=C_t+2\mu I t- \frac12 \int_0^t H_s^TA_{(-\infty,s)}^{-1} H_s ds.
\end{align*}
Then $(F_t, G_t, t\ge 0) \eqd (B_t,C_t,\ge 0)$. 
\end{theorem}

Here the analogy to the one-dimensional case is not as immediate, but the  process $t \mapsto 2 \mu t I -  \frac12  \int_0^t H_s^TA_{(-\infty,s)}^{-1} H_s ds$
can be seen to correspond to $r_0- r_t $ by  first differentiating and then integrating back up in the definition \eqref{TITO1}. Note as well that $H_{t/2} \eqd M_{t}^{(\mu)}$.

Theorems \ref{thm:OIOO} and \ref{thm:TITO} are proved in Section \ref{sec:ProcessDufresne}. 

\subsection{Connection to spiked random matrices}
\label{sub:IntroSpike}

An important problem in mathematical statistics is to describe the law of the largest eigenvalue of sample covariance (or Wishart) matrices of the form $G \Sigma G^{\dagger}$. In the basic setting  $G$ is $p \times q$ and comprised  independent unit Gaussians in $\mathbb{F} = \R, \mathbb{C}, $ or $\mathbb{H}$ and $\dagger$ is the associated conjugate transpose. One is typically interested in the limit as $p$ and $q$ tend to $\infty$ with $\Sigma$ some deterministic sequence of symmetric population matrices. When $\Sigma = I$, this so-called soft-edge limit is well known to be given by the $\beta=1,2$ or $4$ Tracy-Widom laws (for the case of   $\R, \mathbb{C}, $ or $\mathbb{H}$ entries respectively). Moving toward the more general problem, the spiked ensembles in which $\Sigma = \Sigma_r \oplus I_{q-r}$ and $r$ remains fixed as $p$ and $q$ grow have generated considerable interest.

Using the determinantal framework at $\beta=2$, \cite{BBP} proved there exists a phase transition.  Below criticality one sees Tracy-Widom in the limit, above criticality there  are Gaussian effects (the limit given by the law of the largest eigenvalue of a finite rank GUE), with a new one parameter family of spiked soft-edge laws in the crossover regime.    Subsequent analytic work was carried at $\beta = 1$ and $\beta = 4$ by \cite{Mo} and \cite{Wang}, among others.
 
 In another direction, \cite{BV1, BV2} proved that the $\beta=1,2,$ or $4$ soft-edge spiked laws can  be characterized in a unified way
through the eigenvalue problem for the (random) operator $H$ acting on functions $f \in L^2[ [0,\infty), \mathbb{F}^r]$ defined by
$$
  H = - \frac{d^2}{dt^2} + r t + \sqrt{2} \mathcal{B}_t^{\prime}, \qquad f^{\prime}(0) = C f(0).
$$
Here $\mathcal{B}_t$ is the standard $\mathbb{F}$-invariant Brownian motion, that is, for $U \in U_r(\mathbb{F}) =
\{ V \in \mathbb{F}^{r \times r} : V V^\dagger = I \}$ it holds that $U \mathcal{B}_t U^{\dagger} \eqd  \mathcal{B}_t$,  and $C$ is the scaling limit of the matrix $\Sigma_r$. At $r=1$ the result holds for all $\beta > 0$. In that case the noise term reduces to $\frac{2}{\sqrt{\beta}} b_x^{\prime} $, and $H$ is recognized as the Stochastic Airy Operator from \cite{RRV}, but with a Robin (rather than Dirichlet) boundary condition.

The authors of \cite{RR2} asked whether one could similarly spike the hard edge, or smallest eigenvalue laws for general  $\beta$ (though see \cite{DF} for earlier work specific to $\beta=2$).  This regime is defined by setting $q = p+a$ for $a >-1$ remaining fixed as $p \rightarrow \infty$. 
A primary motivation was to confirm that the resulting spiked hard edge laws recover all known spiked soft edge laws via the  familiar hard-to-soft transition. 
 
To describe the spiked hard-edge operator, we set $\mu  = \frac{a+r}{2}$
and introduce
\begin{equation}
\label{MStrange}
  d M_t =   M_t dW_t + \left(  \frac{1}{2 \beta}  - \mu \right) M_t dt, \qquad  \mathcal{M}_t = M_t M_t^{\dagger} ,
\end{equation}
with the usual $M_0=I$. Here $t \mapsto W_t$ is again an  $r \times r$ matrix of independent Brownian motions, but now 
with each off-diagonal entry a unit 
$\mathbb{F}$-valued Brownian motion and each diagonal entry a real Brownian motion with mean-square zero and mean-square $\beta^{-1}$.
The relevant result from  \cite{RR2} is that, under the condition that $ (p \Sigma_r)^{-1} \rightarrow C^{-1}$ in norm, the limiting smallest  eigenvalues of $ p G \Sigma G^{\dagger} $ are described by the eigenvalue problem for: 
\begin{align}
\label{MatrixOp}
  G f(t) =   \int_0^{\infty} \left(  \int_0^{t \wedge  s}  e^{ru}{{\mathcal{M}}}_u^{-1} du \right)  {\mathcal{M}}_s f(s) ds 
     +    C^{-1} \int_0^\infty  {\mathcal{M}}_s f(s) ds, 
\end{align}
acting on  $f \in L^2[\mathcal{{M}}]$,  the space of  functions $f : \mathbb{R}_+ \mapsto \mathbb{F}^r$ with  
 $   \int_0^\infty (f^\dagger \mathcal{M}_t f)(t) dt < \infty$.  The operator $G$ is positive compact and actually describes the limiting {\em{inverse}}
 Wishart eigenvalues, that is, the limiting Wishart eigenvalues are  the spectral points $\lambda$ for the problem $ \lambda G f = f$.  

As in the soft edge case, when $r=1$ the result is valid for all $\beta >0$. Also when $r=1$, both $\mathcal{M}_t$ and 
$e^{rt} \mathcal{M}_t^{-1}$   reduce to geometric Brownian motions and $G$ has the interpretation of the Green's function for a Brownian motion in a Brownian potential. And again  at $r=1$ with $C = c \in \R$, the limit $c \rightarrow \infty$ recovers the basic $\beta$ hard edge operator introduced in \cite{RR1}. 

While there is no critical point at the hard edge,
 the supercritical regime refers to choosing  $C = c I$ and taking $c \rightarrow 0$ in \eqref{MatrixOp}.
At one level the outcome is easy to describe: $c G_{ cI}$ converges (almost surely in operator norm) to the finite rank operator defined by integration against $\mathcal{M}_t$. By analogy with the $r\times r$ Gaussian invariant ensemble  supercritical limit at the spiked soft edge, the obvious conjecture was that 
$ \mathrm{spec} ( \int_0^{\infty} \mathcal{M}_t dt ) $ should be described by the inverse Wishart law(s).  And of course for $r=1$ 
the conjecture was known to be correct due to the original Dufresne identity.

\begin{corollary}  
\label{cor:ComplexCase} With the appropriate choice of $t \mapsto W_t$,
the eigenvalues of the $r \times r$ random matrix $\int_0^{\infty} \mathcal{M}_t dt$ have  joint law given by the eigenvalues of  the inverse $\mathbb{F}$-Wishart distribution with parameter $2 \mu$.   Using isotropic matrix Brownian motions  and 
 replacing the underlying process in \eqref{MStrange} by
\begin{equation}
\label{MFlat}
  d M_t =  M_t dB_t + \left(  \frac{1}{\beta}  - \frac{1}{2}  - \mu \right) M_t dt, \quad   M_0 = I,
\end{equation}
with $B_t$ now comprised completely of independent unit 
$\mathbb{F}$-valued Brownian motions, the full  $\mathbb{F}$-Wishart distribution is recovered by
$ \int_0^{\infty} M_t M_t^{\dagger} dt $.
In both cases, the natural condition on $\mu$ remains $2 \mu > r-1$.
\end{corollary} 

Of course, when $\beta=1$ the equations \eqref{MStrange} and \eqref{MFlat} agree and the above is a repeat of Theorem \ref{thm:MatrixDufresne}.  While the structured noise in \eqref{MStrange} is what arises in the spiked random matrix problem, we mention the result for \eqref{MFlat} for  $\beta=2$ and $4$ as it seems a more natural construction and readily produces the full matrix law. The proof of  Corollary \ref{cor:ComplexCase} is sketched alongside the proof of Theorem \ref{thm:MatrixDufresne} in Section \ref{sec:DufresneBasic}.

\subsection{Further questions}
\label{sub:IntroOpen}

The most obvious question is whether exists a (solvable) polymer model in matrix variables. The semi-directed Brownian polymer (or O'Connell-Yor polymer) alluded  to above, can be defined by the partition function
$$
   \mathcal{Z}_{n,t} = \int_{0 < s_1 < \cdots s_{n-1} < t} ds_1 \dots ds_{n-1}
                        \exp{ ( b^{(1)}_{(0, s_1)} + b^{(2)}_{(s_1, s_2)} +  \cdots + b^{(n)}_{(s_{n-1}, t)} )},
$$
where $(b^{(1)}, \dots, b^{(n)})$ is a standard  Brownian on $\R^n$. The  stationary version defined earlier in \cite{OconYor} has the first level not started at zero, but instead distributed over the negative half-line by the measure 
$e^{b_t - \mu t}$.  That partition function was in fact arrived at, and aspects of its law understood, by iterating the Brownian Burke property  described in Section \ref{sub:IntroBurke} (recall \eqref{TITO1} and \eqref{TITO2}). What is missing in our case is a matrix Burke property that can be iterated through the non-commutativity in the same fashion.

In \cite{Oconnell1} O'Connell shows that the  $t \mapsto \mathcal{Z}_{n,t}$ process has the same law as the top component of a diffusion on $\R^n$ whose generator is a conjugation of the quantum Toda Hamiltonian. Remarkably, when $n=2$ this result is exactly the Matsumoto-Yor $2M-X$ theorem ($z_t$ is $\mathcal{Z}_{2,t}$ up to a change of variables). In both cases, there is an intertwining (between $\mathcal{Z}_{n,t}$ or $z_t$ and the driving $\R^n$ or $\R$ Brownian motion) which provides a fairly explicit formula for the Laplace transform  of the ``$z$" processes.
  While we have an analogous intertwining (Corollary \ref{cor:Intertwining}), the semigroup of Brownian motion on $GL_r$ does not have a sufficiently concrete expression to afford a better characterization of the law of $Z_t$.  Potentially one might be able to bypass the intertwining, and find some description of the joint law $(M_t, A_t)$, and so $Z_t$, by more direct means (again, there are several such routes at $r=1$ \cite{MY2}).

One might also consider various parts of the above program for different groups. In this general spirit, but from different directions, we point out the very recent papers of Chhaibi \cite{Chhaibi} and Bougerol \cite{Bougerol1}.  In the second reference, geometric considerations lead to a process similar in structure to our $Z_t$, but constructed from a Brownian motion on the group of complex lower triangular matrices with positive diagonal.  The singular values of this object are then shown to be Markov with generator given by a conjugation (by a polynomial function  in copies of the Macdonald function and their derivatives) of quantum Toda on a Weyl chamber. Any direct link to the formulas derived here $-$ to the generator of 
$Z_t$ (Theorem \ref{thm:MatrixXandZ}) or that for its singular values (see Lemma  \ref{XandZspec} below) $-$ is not immediately transparent. More simply, it is natural to ask what matrix laws beyond the Wishart can be constructed from of a ``Dufresne procedure" (back in the vein of Theorem 1).

\subsection*{Acknowledgements}

We thank  F.~Baudoin, G.~Let\'ac, and N.~O'Connell for their interest and many helpful discussions.  Thanks as well to D.W.~Stroock for pointers to the PDE literature,  and T.~Kurtz for assistance with the proof of Proposition \ref{prop:Zasymp}.  B.R. was supported in part by NSF grants DMS-1340489 and DMS-1406107, as well as grant 229249 from the Simons Foundation.  B.V. was supported in part by the NSF CAREER award DMS-1053280.

\section{The matrix Dufresne identity}
\label{sec:DufresneBasic}

We prove the basic matrix Dufresne identity in two different ways. Stated above as Theorems \ref{thm:MatrixDufresne} and \ref{thm:MatrixBessel}, they fall below under the headings ``Diffusion" and ``Feyman-Kac" proof. We also provide a sketch of a ``diffusion" proof of Corollary \ref{cor:ComplexCase}.

First  we summarize some of the properties of $M_t = M_t^{(\mu)}$. 
Using the Taylor expansion of the determinant near $I$  
%\[
%\det(I+A)=\sum_{k=0}^\infty \frac{1}{k!}\left(- \sum_{j=1}^\infty \tfrac{(-1)^j}{j} \Tr(A^j)\right)^k, \qquad \textup{if }\|A\|<1.
%\]
and It\^o's formula one finds that,
\begin{align}\label{detM00}
d \det M_t=\det M_t(\Tr dB+r(\tfrac12 +\mu) dt),
\end{align}
for $t\ge 0$. Hence, $\det M_t=\exp(\Tr B_t+\mu r t)$, and  $M_t$ is  almost surely invertible for all time. Then, by the linearity of the sde \eqref{Mt}  it follows that, for any $s>0$, 
the process $t\to M_{s,t}=M_s^{-1} M_{t}, t\ge s$ satisfies the same equation subject to  $M_{t,t}=I$. But that means that 
for fixed $s\ge 0$: 
\begin{align}\label{multincr}
\{M_{s,s+t}, t\ge 0\} \eqd \{M_t, t\ge 0\}, \qquad  \{M_{s,s+t}, t\ge 0\} \textup{ is independent of $\{ M_r, \, r \le s \}$.}
\end{align}
These properties  allow
a natural  extension of  $M_t$ to all $t\in \R$. First extend the matrix Brownian motion $B_t$ for all $t\in \R$, and consider the strong solution $\tilde M_t$ of the sde:
\[
d\tilde M_t=\tilde M_t dB_t+\tilde M_t (\tfrac12 +\mu) dt, \qquad \tilde M_{-1}=I, \quad  t\in [-1,0].
\]
Setting $M_t=\tilde M_0^{-1} \tilde M_t$ for $t\in [-1,0)$, the extended $M_t, t\ge -1$ process satisfies \eqref{multincr} for any $s\ge -1$. Repeating this procedure for earlier starting points defines a version of $M_t$ over the whole line. Importantly,  the resulting process has  properties \eqref{multincr} for each $s \in \R$ and  $t\to M_{s,t}$ satisfies  (\ref{Mt}) for $t\ge s$ with $M_{t,t}=I$.

To close, we state the following lemma on the norm growth of  $M_t$ which we will use repeatedly. The proof is deferred to the very end of the section.

\begin{lemma}
\label{normlemma}
Let $m_1 \le m_2  \le \cdots \le m_r$ be the singular values of $M_t = M_t^{(\mu)}$.   It holds that
$ \lim_{t \rightarrow \infty} \frac{1}{t} \log m_i(t) =   \mu +  \frac{ i -1}{2}$ with probability one for each $i = 1, \dots, r$.
\end{lemma}

\subsection{Diffusion proof}

We actually prove Theorem \ref{thm:MatrixDufresne}  in two different ways as well.  For completeness we first indicate how everything works  directly 
through the matrix coordinates. The proof is somewhat more transparent in eigenvalue/eigenvector 
coordinates, and we carry out that approach afterwards.

\subsubsection*{Via matrix coordinates}

Recall the definition of $M_t=M_t^{(-\mu)}$ with the convergent choice of sign for the drift:
\begin{equation}\label{Magain}
   dM_t = M_t dB_t + ( \tfrac12- \mu) M_t dt, \qquad M_0 = I, \quad t\ge 0
\end{equation}
and of course $2 \mu > r-1$.

Consider the version of this process extended to the whole line (as described just above), and then introduce the time reversed process $N_t = M_{-t}$. We claim that $N_t$ is also a Brownian motion on $GL_r$, but with drift $\mu$ instead of  $-\mu$. In particular, for $t\ge 0$ it solves the SDE
 \begin{align}\label{SDE1}
d N_t= N_t d\tilde B_t+(\frac12+ \mu)  N_t dt , \qquad  N_0=I, \quad t\ge 0,
\end{align}
where $d\tilde B_t =-dB_{-t}$. A quick (but formal) explanation for this statement would follow from  $(I+dB_t+(\tfrac12-\mu)I dt)^{-1}-I\approx  -dB_t-(\tfrac12-\mu)I dt+dB_t dB_t$ and $dB_t dB_t =I dt$. For the precise proof one needs to first verify that $N_t$ also satisfies the stationary and independent increment property as $M_t$, and then to show that  if $M_t$ solves \eqref{Magain} on say $t\in [0,1]$,  the process $\tilde N_t=M_1^{-1} M_{1-t}, t\in [0,1]$ will solve the same sde with $+\mu$ instead of $-\mu$ and with $-dB_{1-t}$ playing the role of $dB_t$.

From this point (with a small abuse of notation) we will drop the tilde from  $d\tilde B_t$ and just use $dB_t$ for the noise in $N_t$.%

Next consider $ N_{s,t}= N_t^{-1} N_{s+t}$ for any fixed $t$. Then $N_{s,t }\eqd N_{s}$, again as processes in $s\in \R$, and  if we further define
\begin{equation}
\label{Qdiffusion}
Q_t:=\int_{-\infty}^0  N_{s,t} N_{s,t}^T ds=N_t^{-1} \left(\int_{-\infty}^t  N_s^{} N_s^T ds \right) N_t^{-T} ,
\end{equation}
we have at last a process that is stationary with marginal law given by that of 
$$
Q_0 =  \int_{-\infty}^0 N_s N_s^T ds \eqd \int_0^{\infty} M_t M_t^T dt.
$$
The proof of Theorem \ref{thm:MatrixDufresne} then comes down to the following.

\begin{proposition} Let $2\mu>r-1$. The process   $t \mapsto Q_t$ is  a diffusion corresponding to the matrix sde
\begin{equation}
\label{basic_Qprocess}
dQ_t=I dt + (1-2\mu) Q_t dt+\Tr Q_t I dt-dB_t Q_t-Q_t dB_t^T.
\end{equation}
If  $Q_0 \in \mathcal{P}$ then  $Q_t$ remains in $\mathcal{P}$ for all $t>0$ with $\gamma_{2\mu}^{-1}$ as  its unique invariant measure. 
\end{proposition}

\begin{proof} 
 Applying It\^ o's formula in \eqref{Qdiffusion} we find that
\begin{align*}
dQ_t=&Idt+ d  N^{-1}_t \left(\int_{-\infty}^t  N_s N_s^T ds \right) N_t^{-T}+
N_t^{-1}  \left(\int_{-\infty}^t  N_s N_s^T ds \right) d N^{-T}_t
\\&+d N^{-1}_t \left(\int_{-\infty}^t  N_s N_s^T ds \right)d N^{-T}_t.
\end{align*}
We also have that, 
\begin{align*}
 d N^{-1}_t=&-N_t^{-1} dN_t^{} N_t^{-1}+N_t^{-1} dN_t N_t^{-1}dN_t N_t^{-1} \\
% =&-(dB+(\mu+ \frac12)I dt)N_t^{-1}+dBdBN_t^{-1} \\
 =&- dB_t^{}  N_t^{-1}+ (\frac12 - \mu) N_t^{-1} dt,
 \end{align*}
 after substituting in \eqref{SDE1} and using $dB_t dB_t = I dt$ in the second term of line one.
Combined, this produces,
\begin{align*}
dQ_t =I dt-(dB+(\mu-\frac12)I dt) Q_t-Q_t(dB+(\mu-\frac12)I dt)^T+dB Q_t dB^T,
\end{align*}
which simplifies to \eqref{basic_Qprocess} on account of the rule  $dB_t C dB_t^T = (\Tr C) I dt$ for any matrix $C$.

To see that $Q_t$, defined by  \eqref{basic_Qprocess}  and a fixed starting point $Q_0  \in \mathcal{P}$, remains in $\mathcal{P}$ for all time, 
apply It\^o's formula yet again, now to $t \mapsto \det (Q_t)$:
\begin{align}
\label{detsde}
 d\det(Q_t) & = \det(Q_t) \left(  \Tr[ Q_t^{-1} dQ_t ] + \frac{1}{2}  (\Tr [ Q_t^{-1} dQ_t] )^2 - \frac{1}{2} \Tr [Q_t^{-1} dQ_t Q_t^{-1} dQ_t]   \right) \\
                   & = \det(Q_t) \left( - 2 \Tr dB_t   + \Tr Q_t^{-1} dt + (2r - 2 \mu ) dt \right). \nonumber
\end{align}
For the second line we used 
$
  \Tr [ Q_t^{-1} dQ_t Q_t^{-1} dQ_t ] = \Tr[  Q_t^{-1} (dB)^2 Q_t] + \Tr  (dB)^2 +  2 \Tr  [Q^{-1} dB Q dB^T].
$
Introducing $b_t = - \Tr B_t /\sqrt{r}$ and 
$$
   dz_t = 2 \sqrt{r} z_t db_t + (2r - 2\mu) z_t dt, \quad z_0 = \det(Q_0) > 0,
$$
we see that  $t \mapsto \det(Q_t)$ is bounded below by the geometric Brownian motion $z_t$ 
up to the first passage of $\det (Q_t)$ to zero. But  $z_t$ never vanishes, and so that passage time must be infinite. 
%That is, $\det (Q_t) >0$ for all $t >0$.

The next ingredients are to show that $Q_t$ admits a smooth positive transition density on $P$, and then to identify the inverse Wishart law $\gamma_{2\mu}^{-1}$ as an invariant measure. Since the latter also has a positive density with respect to Lebesgue measure on $\mathcal{P}$, it will follow that this is the unique invariant measure for $Q_t$. 

The generator of $Q_t$ can be succinctly expressed in matrix coordinates as in,
\begin{equation}
\label{QGenerator}
  G_Q = 2  \Tr  \left( Q^2 (\frac{\partial}{\partial  Q} )^2 \right) 
    + (1- 2 \mu) \Tr \left( Q \frac{\partial}{\partial Q} \right)  + (1 + \Tr  Q) \Tr  \left( \frac{\partial}{\partial Q} \right),
\end{equation}
where again $\frac{\partial}{\partial Q}$ is the matrix-valued operator $[\frac{\partial}{\partial Q}]_{ij} =  ( \frac{1}{2} + \frac{1}{2} \delta_{ij} ) \frac{\partial}{\partial Q_{ij}} $. The second order part of
\eqref{QGenerator} is verified by writing out
$$
  (dB Q + Q dB^T)_{ij} (dB Q + Q dB^T)_{k\ell}
  = [Q^2]_{j\ell} \delta_{ik} + [Q^2]_{i\ell} \delta_{jk} + [Q^2]_{jk} \delta_{i \ell} + [Q^2]_{ik} \delta_{j\ell},
$$
and summing over $i\le j$ and $k \le \ell$. One then checks that the adjoint takes the form,
\begin{align}
\label{QAdjoint}
 G_Q^*  =  \ & 2  \Tr  \left( Q^2 (\frac{\partial}{\partial  Q} )^2 \right) 
    + ( 2 \mu + 2r +3 ) \Tr \left( Q \frac{\partial}{\partial Q} \right)  + ( \Tr Q - 1) \Tr  \left( \frac{\partial}{\partial Q} \right)  \\
    &  +
     \mu r (r+1) +   \frac{1}{2} r (r+1)^2,   \nonumber
 \end{align}
when restricted to act on functions of a symmetric matrix variable.

To invoke the necessary  regularity estimates we temporarily consider the ``vectorized" $Q_t$, or 
$\vect(Q_t) = (Q_{11}(t), Q_{22}(t), \dots ) \in \R^{\frac{r(r+1)}{2}}$. In particular, we show that the diffusion matrix written in these coordinates is positive definite on the open set $\mathcal{P} \subset \R^{\frac{r(r+1)}{2}}$.  This comes down to proving: if
 $Q$ is a positive  definite matrix and $Z$ is an $r\times r$ matrix normal with iid entries then the covariance matrix of $\vect(ZQ+QZ^T)$ is positive definite. 
To see this, take any  $r\times r$ orthogonal matrix $O$ and let $K_O$ be the linear map on $\R^{\frac{r(r+1)}{2}}$ defined by $\vect(O^T Q O)=K_O \vect(Q)$. It is easy to verify that $|\det K_O|=1$. Next write the spectral decomposition of $Q=O^T \Lambda O$, and note that
 the desired covariance matrix satisfies: 
\begin{align*}
& \ev \vect(ZQ+QZ^T) \vect(ZQ+QZ^T)^T\\
%&\qquad =E \vect(ZO^T \Lambda O+O^T \Lambda OZ^t) \vect(ZO^T \Lambda O+O^T \Lambda OZ^t)^t\\
%&\qquad =E \vect(O^T (Z\Lambda + \Lambda Z^t)O) \vect(O^T (Z\Lambda + \Lambda Z^t)O)^t \\
&\qquad = \ev K_O \vect(Z\Lambda + \Lambda Z^T) \vect(Z\Lambda + \Lambda Z^T)^T K_O^T.
\end{align*}
Thus, by the first remark, we may assume that $Q$ is diagonal with entries $Q_{ii}>0$. But in that case the covariance matrix is diagonal with entries $4 [Q^2]_{ii}, 1\le i\le r$ and $[Q^2]_{ii}+[Q^2]_{jj}, 1\le i<j\le r$ which is clearly positive definite. 

At last then,  Theorem 3.4.1 of \cite{Stroock} (though see also  Remark 3.4.2 there) implies that $\partial_t - G_Q^*$ is hypoelliptic on $\mathcal{P}$.  At the same time a straightforward but tedious calculation will show that $G_Q^* f(Q) = 0 $ for $f(Q)  = (\det Q)^{-\mu -\frac{r+1}{2}} e^{-\frac12 \Tr Q^{-1}}$, the $\gamma_{2\mu}^{-1}$ density. For smooth test functions $h$ of compact support on $\mathcal{P}$ and $T_t$ the semigroup  of $Q_t$, 
we have
$$
     \frac{\partial}{\partial t} \int_{\mathcal{P}} (T_t h)(Q) f(Q) dQ  = \int_{\mathcal{P}} G_{Q} (T_t h)(Q) f(Q) dQ.
$$
An integration by parts would  continue the equality as $\int_\mathcal{P}  (T_t h)(Q) (G_{Q}^* f)(Q) dQ = 0$ and complete the proof. To justify this, that is, that there are no boundary terms, requires two facts. The first is that $T_t h(Q)$ and its normal derivative are bounded along the boundary $\det(Q)=0$. This can be established by writing $G_{Q}$ in local coordinates in the vicinity of $\det(Q)=0$, and working by comparison with one-dimensional process \eqref{detsde} whose semi-group 
is readily seen to have the desired property at the origin. The second is to check by a simple computation that $\frac{\partial}{\partial Q_{ij}} f(Q)  |_{\det(Q)=0} = 0$ for all $i,j$, noting that $f(Q) |_{\det(Q) = 0} = 0$ is obvious.
\end{proof}

\subsubsection*{Via the eigenvalue law}

One could alternately argue that, since the law of  $Q_t$ defined in \eqref{basic_Qprocess} is invariant under rotations by the orthogonal group, it is enough to consider the motion of the eigenvalues. More convenient still, is to identify the Wishart law $\gamma_{2\mu}$ itself  by considering instead the eigenvalues of $P_t = Q_t^{-1}$. 

We have that,
\begin{equation}
\label{QInverse}
 d P_t = - P_t^2 dt+ (1+2 \mu) P_t dt+ \Tr P_t I dt - dB_t^T P_t - P_t dB_t,
\end{equation}
the solution of which, by similar reasoning as above, also remains in $\mathcal{P}$ for all time after starting from any point in the interior.
Further, $p_r \ge p_{r-1} \ge \cdots \ge p_1 \ge 0$, the ordered eigenvalues perform the joint diffusion,
\begin{equation}
\label{wishart_eigs}
   d p_i = - p_i^2 dt +(2\mu+2) p_i dt +  p_i  \sum_{j \neq i}  \frac{p_i + p_j}{p_i - p_j}  dt + 2 p_i db_i,
\end{equation}
with $\{ b_i \}_{i=1, \dots, r}$ independent standard Brownian motions. The above can be derived from \eqref{QInverse} by computing the It\^o differential of the corresponding spectral representation.
  We will do a sample of such a calculation below in a slightly more complicated context. It is by now standard that system \eqref{wishart_eigs} possesses a strong solution, that the paths $p_i = p_i(t)$ do not intersect for $t > 0$, and any initial condition $P_0$ with some $p_i(0) = p_{i+1}(0)$ is an entrance point, see for example \cite[\S 4]{AGZ}. 
 
The action of the corresponding generator $G_{p}$ can be expressed in the form,
 \begin{align}
 \label{Peig_Generator}
 G_p f
 & = \sum_{i=1}^r 2 p_i^2 \partial_i^2 f + 
      \sum_i \left(- p_i^2 +(2\mu+2) p_i +  p_i  \sum_{j \neq i}  \frac{p_i + p_j}{p_i - p_j} \right) \partial_i f  \\
  & = \sum_{i=1}^r  \partial_i( \phi( p_i) \partial_i f) - ( \phi( p_i) \partial_i V)  \partial_i f,  \nonumber
 \end{align}
 where, 
 \begin{equation}
 \label{WishartPotential}
   \phi(p_i)  = 2 p_i^2, \qquad 
   V(p)= - \sum_{i> j} \log(p_i-p_j) +  \frac{1}{2} \sum_{i=1}^r    p_i -    (\mu - \frac{r+1}{2} )  \sum_{i=1}^r \log p_i,
\end{equation}
restricted to the Weyl chamber $\Sigma_r = \{ p_r  \ge \cdots \ge p_1 \ge p_0=0 \} \subset \R^r$.
Now $e^{-V(p)} \ind_{\Sigma_r}(p)$ is recognized as (after a suitable normalization) the joint density of eigenvalues for the real Wishart ensemble $\gamma_{2\mu}$, and the form of $G_p$ in line two of \eqref{Peig_Generator} is particularly suited for verifying that
$ G_p^* (e^{-V(p)} ) =0$.

To identify $e^{-V(p)}$ as the invariant measure, we only have to deal with the same integration by parts issue that came up when working with matrix coordinates. In this case Appendix A of \cite{ESY} explains why for smooth $h$ of compact support in $\Sigma_r$ we have that $x \mapsto \ev_x [ h(p_1(t), \dots, p_r(t) ) ]$ along with its normal derivative are bounded at the seams $p_{i+1}=p_{i}$ (or boundary of $\Sigma_r$). After that a quick calculation yields $  \lim_{p_{i+1} \rightarrow p_i} \frac{\partial}{\partial (p_{i+1}+p_i)} e^{-V(p)}  = 0$.

\subsubsection*{Proof of Corollary \ref{cor:ComplexCase}}

For the complex and quaternion cases it is a bit more constructive to go through eigenvalue/eigenvector coordinates. The starting point remains the same: a matrix diffusion $t \mapsto Q_t$ is constructed which has the desired distribution as its invariant measure (assuming the latter exists and is unique). 

Corollary \ref{cor:ComplexCase} considers two setting. The analogs of \eqref{basic_Qprocess} are:  
\begin{equation}
\label{flatQ}
  d Q_t  =    I dt + (\frac{2}{\beta} - 1 - 2 \mu) Q_t dt  + \Tr Q_t dt  + dB_t Q_t  + Q_t dB_t^{\dagger},  
\end{equation}
for the $U_r(\mathbb{F})$-invariant noise $t \mapsto B_t$, and 
\begin{equation}
\label{funnyQ}
  d Q_t  =    I  dt +  (\frac{1}{\beta} - 2 \mu ) Q_t dt    + (\frac{1}{\beta} -1) \mbox{diag}(Q_t) dt  + \Tr Q_t dt
                + d W_t Q_t  + Q_t d W_t^{\dagger},
\end{equation}
for the case of the particular structured $\beta=1,2,4$ noise $t \mapsto W_t$ appearing in the original spiked random matrix problem.
Here $ \mbox{diag}(Q_t) $ is the diagonal matrix with the same diagonal as $Q_t$. This corresponding term in the sde is shows up because of 
 $dW_t  Q_t dW_t^{\dagger}=(\frac{1}{\beta} -1) \mbox{diag}(Q_t) dt  + \Tr Q_t dt$.

The key observation is: 
 
\begin{proposition}
\label{Prop:commoneigs}
In either setting \eqref{flatQ} or \eqref{funnyQ}, the eigenvalues of $P_t  = Q_t^{-1}$ are Markovian with common sde:
\begin{equation}
\label{CommonEigSDE}
d p_i =   - p_i^2  dt   + ( 2 \mu + \frac{2}{\beta}) p_i dt + p_i \sum_{j \neq i} \frac{p_i + p_j}{p_i- p_j}  dt + \frac{2}{\sqrt{\beta}} p_i db_i,
\end{equation}
for $\beta=1,2,$ or $4$.
\end{proposition}

Compare \eqref{wishart_eigs}, noting the overlap at $\beta=1$.
That the eigenvalues of \eqref{flatQ} are Markov is self-evident. For \eqref{funnyQ} we only see it by going through the calculation, which we defer to the end of the proof. Note that one can use \eqref{CommonEigSDE} after that fact to see that $P_t$ (and so $Q_t$) remains in $\mathcal{P}$ for all time.

The upshot is that, for the eigenvalue motion(s), the argument is now precisely the same as in the $\beta=1$ case. The corresponding 
$\beta =2$ or $4$ generator $G_{\beta,p}$ has the same form as \eqref{Peig_Generator}, with \eqref{WishartPotential} replaced by:
$$
  \phi(p_i)  = \frac{2}{\beta} p_i^2, \qquad 
   V(p)= -  \beta \sum_{i> j} \log(p_i-p_j) +  \frac{\beta}{2} \sum_{i=1}^r    p_i -   \beta  
   \left( \mu - \frac{r+ 2/\beta- 1}{2}  \right)  \sum_{i=1}^r \log p_i.
$$
It follows that $G_{\beta, p}^* (e^{-V(p)}) = 0$.  And, as it has to be,  
 $e^{-V(p)} \ind_{\Sigma_r}(p)$ is proportional to the complex/quaternion  Wishart eigenvalue density. 
For the isotropic setting one then has the full $\mathbb{F}$-Wishart law, as the eigenvector process of \eqref{flatQ} clearly has the Haar measure on $U_r(\mathbb{F})$ as its unique invariant measure.

\begin{proof}[Proof of Proposition \ref{Prop:commoneigs}]
It\^o's formula shows that if $Q$ solves \eqref{flatQ} or \eqref{funnyQ} then $P$ solves the SDE analogue to \eqref{QInverse}, with $\dagger$ in place of the transpose and an extra $(\tfrac1{\beta}-1) \textup{diag}(P) dt$ term in the second case. 

At this point we make two simplifications. For clarity we carry out the computation for $\beta=2$ only (that the $\beta=4$ case will go through in the same way will be clear). Also, we consider the simplified matrix sde:
\begin{equation}
\label{newfunny}
   d  P_t =  F(P_t) dt   + d W_t P_t  + P_t d W_t^{\dagger}, \quad F(P) = -\frac12 \mbox{diag}(P), 
\end{equation}
where again $W_t$ has independent real Brownian motions with variance $\frac12 t$ on the diagonal and independent 
unit complex Brownian motions elsewhere. The point is that \eqref{newfunny} retains everything ``non isotropic" in \eqref{funnyQ}. That 
the corresponding isotropic case  (with $B_t$ replacing $W_t$ and no $\mbox{diag}(\cdot)$ term in the drift) produces the same answer will also become clear in the course of the proof.

Either way, the strategy is standard. Write
$P_t = U_t^{\dagger} \Lambda_t U_t$  for $\Lambda_t$ the diagonal matrix of eigenvalues $(\lambda_{1,t}, \dots, \lambda_{r,t})$ and
a unitary matrix $U_t$. Also introduce the notation,
\begin{equation}
\label{rules1}
   \mathsf{W}_t = U_t  W_t U^{\dagger}_t, \qquad  \quad d U_t U_t^{\dagger} = d\Gamma_t + dG_t.
\end{equation}
Note that the former is that it is not simply a copy of $W_t$.
The latter is the Doob-Meyer decomposition, with $\Gamma_t$ a local martingale and $G_t$ of  finite variation. 
Since $d (U_t U_t^{\dagger}) = 0$ one finds that
\begin{equation}
\label{NGRules}
   d \Gamma_t^{\dagger} = - d \Gamma_t, \quad \mbox{ and } \quad dG_t+dG_t^{\dagger} = - d\Gamma_t d\Gamma_t^{\dagger}, 
\end{equation}
in particular $d\Gamma_{ii}=0$. 
With this in hand an application of It\^{o}'s formula produces
\begin{align}
\label{EigSystem}
 d \Lambda =  & \, d \mathsf{W} \Lambda   +  \Lambda d {\mathsf{W}}^{\dagger} +  U F(U^{\dagger} \Lambda U )  U^{\dagger}dt + 
 (d \Gamma  \Lambda+  \Lambda d \Gamma^\dagger)
     +  (d G  \Lambda + \Lambda d G^\dagger) \\
       & \,  + d \Gamma \Lambda d \Gamma^{\dagger} + d\Gamma d \mathsf{W} \Lambda + d\Gamma \Lambda d \mathsf{W}^{\dagger} + d \mathsf{W} \Lambda d \Gamma^{\dagger} 
             +  \Lambda d \mathsf{W}^{\dagger} d \Gamma^{\dagger}.  \nonumber
\end{align}
As the martingale part of the right hand side must vanish off the diagonal we infer that,
\begin{equation}
\label{TheN}
    d\Gamma_{ij}
    % = \frac{  p_j  d  {\mathcal{W}}_{ij}  + p_i d {\mathcal{W}}^{\dagger}_{ij}}{p_i - p_j}
                   =  \frac{  \lambda_j d  {\mathsf{W}}_{ij}  + \lambda_i d \overline{ {\mathsf{W}}}_{ji}}{\lambda_i - \lambda_j}.
\end{equation}      
for $i \neq j$. Next, we write out \eqref{EigSystem} on the diagonal:
\begin{align}
\label{EigSystemA}
    d \lambda_i   =  &   \,    \lambda_i ( d  {\mathsf{W}}_{ii} +  d  \overline{\mathsf{W}}_{ii})  + \sum_{j} u_{ij} F_{jj}(U^{\dagger} \Lambda U) \bar u_{ij} dt \\
                                  & \, - \sum_{j \neq i}  (\lambda_i - \lambda_j) d \Gamma_{ij}  d \overline{ \Gamma}_{ij} 
                                   \nonumber \\
                                    &  + \sum_{j\neq i} ( \lambda_i d  {\mathsf{W}}_{ji} + \lambda_j  d \overline{  \mathsf{W}}_{ij}) d\Gamma_{ij}
                                     \,  + \sum_{j\neq i} ( \lambda_i d \overline{  \mathsf{W}}_{ji} + \lambda_j  d { {\mathsf{W}}_{ij}} ) 
                                       d \overline{ \Gamma}_{ij}, \nonumber
\end{align}
having used \eqref{NGRules}. We also record that,
\begin{align}
\label{diagdrift}
 \sum_{j} u_{ij} F_{jj}(U^{\dagger} \Lambda U) \bar u_{ij}  & = -\frac{1}{2}
  \sum_{j = 1}^{r}  
  \sum_{\ell=1}^r 
  u_{ij} \bar u_{\ell j} \lambda_\ell u_{\ell j} \bar u_{ij}=
  -\frac{1}{2}
  \sum_{j = 1}^{r}  
  \sum_{\ell=1}^r  
   \lambda_\ell
     | u_{i j} |^2 | u_{\ell j}|^2, \\
                                     & =  - \frac12 \lambda_i + \frac{1}{2}  \sum_{j \neq i} (\lambda_i - \lambda_j)  \sum_{\ell=1}^r  
     | u_{i \ell} |^2 | u_{j \ell}|^2, \nonumber
\end{align}
where $u_{ab}$  are the entires of $U$.
 
To finish, first note that 
$
    (  {\mathsf{W}}_{ii}+\overline{\mathsf{W}}_{ii} , i=1,\dots, r) \eqd  (\sqrt{2}  b_i, i=1,\dots,r)
$
for independent standard real Brownian motions $b_i$.
 Next, one may check that for $i \neq j$ 
 \begin{align}\label{newnoise1}
 d {\mathsf{W}}_{ij}  d \overline{ {\mathsf{W}}}_{ij} 
  %=   d \tilde{\mathcal{B}}_{ki} \overline{d \tilde{\mathcal{B}}}_{ki} 
  = dt- \frac{1}{2} \sum_{\ell=1}^r 
     | u_{i \ell} |^2 | u_{j \ell}|^2 dt,      \qquad
    d {\mathsf{W}}_{ij} {d \mathsf{W}}_{ji}  =
    %=   \overline{d \tilde{\mathcal{B}}}_{ik} \overline{d \tilde{\mathcal{B}}}_{ki} = 
    \frac{1}{2} \sum_{\ell=1}^r 
     | u_{i \ell} |^2 | u_{j \ell}|^2 dt.
\end{align}
These rules  may then be used along with \eqref{TheN}  in \eqref{EigSystemA} to find that
the contribution of the final two lines 
there is equal to:
\begin{align}\nonumber
   \sum_{j \neq i} (\lambda_i - \lambda_j)^{-1}  &  \Bigl( \lambda_i^2 | d {\mathsf{W}}_{ji}|^2 + \lambda_j^2 | d {\mathsf{W}}_{ij}|^2
                                   + \lambda_i \lambda_j (  d {\mathsf{W}}_{ij} {d {\mathsf{W}}_{ji}} +   d\overline{ {\mathsf{W}}}_{ij} 
                                    d \overline{ {\mathsf{W}}}_{ji} ) \Bigr) \\
                                  =  &  \sum_{j \neq i}  \frac{ \lambda_i^2  + \lambda_j^2}{\lambda_i - \lambda_j} - \frac{1}{2} 
                                   \sum_{j \neq i} (\lambda_i - \lambda_j)  \sum_{\ell=1}^r 
     | u_{i \ell} |^2 | u_{j \ell}|^2.\label{interaction1}
\end{align} 
The last term here now cancels with the last term in \eqref{diagdrift}. This produces the system
$ d \lambda_i = \sqrt{2} \lambda_i db_i - \frac{1}{2} \lambda_i dt  + \sum_{j \neq i}  \frac{ \lambda_i^2  + \lambda_j^2}{\lambda_i - \lambda_j} dt$.
Now putting  back the additional drift terms ($\Tr P - P^2 + (\frac12 - 2 \mu) P$) will produce \eqref{CommonEigSDE} and
 complete the proof.
 
 Note also that in case of the isotropic noise the cross variations in \eqref{newnoise1} simplify as $dt$ and $0$ respectively, and thus in \eqref{interaction1} only the term $ \sum_{j \neq i}  \frac{ \lambda_i^2  + \lambda_j^2}{\lambda_i - \lambda_j} $ remains (as it should). 
\end{proof}

\subsection{Feynman-Kac proof}

We  prove Theorem \ref{thm:MatrixBessel}. 

Taking $M_t$ with the convergent choice of drift, as in \eqref{Magain} and again with $2\mu > r-1$,    the task is to show that: with $Y_t = M_t M_t^T$, 
\begin{equation}
\label{YExpectation}
   \ev \left[ e^{- \frac12 \int_0^{\infty} Y_t dt  } \,  | \, Y_0  = Y \right] = \frac{ K_{r} ( -\mu |   Y,I) }{2^{\mu-1} \Gamma_r(\mu)},
\end{equation}
via the partial differential equation that characterizes the left hand side. 
%Note we have used the relation $K_r(\mu | I,Y)  = K_r(-\mu | Y, I)$
%on the right hand side (compare the statement of Theorem \ref{thm:MatrixBessel}).

 This   relies on the Markov property of $ t \mapsto Y_t \in \mathcal{P}$. Applying It\^o's formula gives 
\begin{equation}
\label{YIto}
  d Y_t = M_t (dB_t + dB_t^T) M_t^T + (r +1  - 2 \mu) Y_t dt,
\end{equation}
 which does not appear to close. But a check of the matrix entry covariances  produces the generator,
\begin{equation*}
\label{GY}
   G_Y = \sum_{i \le j} \sum_{k \le \ell} ( y_{ik} y_{j \ell} +y_{i \ell} y_{jk})  \frac{\partial^2}{\partial y_{ij} \partial y_{k \ell}} + (r+1-2 \mu) \sum_{i\le j} y_{ij}  \frac{\partial}{\partial y_{ij}},
\end{equation*}
which can then be put into the abbreviated form \eqref{YGenerator}:
 $G_Y = 2 \Tr (Y \frac{\partial}{\partial Y} )^2-2\mu \Tr (Y  \frac{\partial}{\partial Y} ).$ 
 (One can also argue that by the rotation invariance of $dB_t$ each appearance of $M_t$ in \eqref{YIto} can be replaced by a positive square root $\sqrt{Y_t}$.)

The standard martingale argument used to derive the Feyman-Kac formula  shows that the left hand side $:= U(Y)$ of \eqref{YExpectation} solves
$ (G_Y - \frac12 \Tr Y) U(Y) = 0$. We have $U(0)=1$ since $Y_t =0$ for all time if $Y_0 = 0$.  
As for for uniqueness, we claim that if $\tilde U(Y)$ is bounded and satisfies 
$ (G_Y - \frac12 \Tr Y) \tilde U(Y) = 0$ and $\tilde U(0) =0$, then $\tilde U(Y)$ is identically zero. Consider the bounded martingale
$
  S_t^Y  = \tilde U(Y_t) e^{-\frac12 \int_0^t  \Tr(Y_s ) ds }.  
$
By Lemma \ref{normlemma} the matrix norm of $Y_t \rightarrow 0$ almost surely as $t\rightarrow \infty$. Hence $S_t^Y$ also converges to zero as $t \rightarrow \infty$. But then $S_t^Y \equiv 0$ as $S_t^Y$ is a regular martingale, and so $\tilde U(Y) \equiv 0$.

We are left to check that the $K$-Bessel function on the right hand side of \eqref{YExpectation} satisfies the PDE.
First note,
$$
   e^{\frac{1}{2} \Tr (xy)}  (2 G_y- \Tr(y)  ) e^{-\frac{1}{2} \Tr (xy)} 
     =  \Tr (xyxy) + (2 \mu-r - 1) \Tr (xy) - \Tr(y),
$$
where from here  on we revert to lower case for matrix variables.
Differentiating under the integral defining $K_{r}(- \mu |   y,I)$,  we are left to show that
\begin{equation}
\label{vanishing}
   0 = \int_{\mathcal{P}} ( \Tr (xyxy) + (2 \mu-r - 1) \Tr (xy) - \Tr(y)) 
   \det(x)^{-\mu} e^{-\frac{1}{2} \Tr  (xy + x^{-1} ) } d \mu_r(x).
\end{equation}
Here we have introduced the  invariant measure on $\mathcal{P}$:
\begin{equation}
\label{InvMeasure1}
d\mu_r(x)= (\det x)^{-\frac{r+1}{2}} dx,
\end{equation}
so called because
\begin{equation}
\label{InvMeasure2}
   z=t^T x t \leadsto  d\mu_r(z) =d\mu_r(x), \qquad z = x^{-1}\leadsto d\mu_r(z)=d\mu_r(x),
\end{equation}
with any invertible $t$ in the first relation,  see \cite{Terras}.  Unfortunately, integrating back up inside the integral  \eqref{vanishing} appears cumbersome. Instead we rely on the uniqueness of the associated Mellin  transform.

\begin{proposition} {\bf{(See Theorem 1 of \cite[\S 4.3.1]{Terras})}}
Introduce the power function,
\begin{equation}
\label{powerfunction}
    p_s(y) = \prod_{k=1}^r \det( y_k )^{s_k} \quad  s = (s_1, \dots , s_r) \in \CC^r,
\end{equation}
where $y_k \in \mathcal{P}_k$ (the $k \times k$ positive definite matrices) is the $k^{th}$ minor of $y \in \mathcal{P} =\mathcal{P}_r$.
Then the Mellin transform,
$$
  \hat{h}(s) := \int_\mathcal{P} p_s(y) h(y) \mu_r(dy),
$$
defines an invertible map on the subspace of rotation invariant functions $-$ those $h$ for which  $h(x) = h(O^T x O)$ for orthogonal $O$ $-$   of $L^2(\mathcal{P}, \mu_r)$.
\end{proposition}

The necessary isometry is actually stated in the cited theorem for the more general Helgason-Fourier transform on $\mathcal{P}$,
but this reduces to the above Mellin transform on rotation invariant functions. The right hand side of
\eqref{vanishing} is clearly rotation invariant in $y$ $-$ it is a function of the eigenvalues of $y$ alone.

Setting 
\begin{equation}
    V(y) =   \int_P  \det(x)^{-\mu} e^{-\frac{1}{2} \Tr  (xy +  x^{-1} ) } d \mu_r(x)= \int_P  \det(x)^{\mu} e^{-\frac{1}{2} \Tr  (x^{-1}y +  x ) } d \mu_r(x),
\end{equation}
that is, $V(y) =  2 K_r(-\mu |  y,I)=2 K_r(\mu | I, y)$, we compute first $\hat{V}(s)$ and then track the multipliers to this result produced after multiplying by, or  integrating against, the additional factors $\Tr (y)$, $\Tr (\cdot \, y)$, and   $\Tr( \cdot \, y \cdot y)$.

\medskip

{\bf Step 1: } This formula can be found in \cite{Terras}, but it guides the later computations so we record it here. The trick is to introduce the change of variables $x = t^T t$ for $t$ upper triangular with $t_{i,j}\in \R,$ $t_{ii}\in \R_+$. This results in the rules: 
\begin{equation}
\label{trick1}
   d\mu_r(x)=2^r \prod_{i=1}^r t_{ii}^{-i} \prod_{i\le j} dt_{ij}, \qquad p_s(x)=\prod_{j=1}^r t_{jj}^{r_j}, \mbox{ for } r_j=2(s_j+\dots+s_n).
\end{equation}
Using  two such changes of variables $y=t^T t$ and  $x=q^T q$ produces
\begin{align}
\label{Mellin1}
 2^{-2r} \hat{V}(s) &
 = \int\int p_s(t^T t) \det (q^Tq)^\mu e^{-\frac12  \Tr q^T q-\frac12 \Tr  q^{-T}  t^T t  q^{-1} }  \prod_j t_{jj}^{-j} q_{jj}^{-j} \,  dt \,  dq \nonumber
\\
%&=2^r\int\int p_s(q^T y q) \det (q^T q)^\mu e^{-\frac12  \Tr q^Tq-\frac12 \Tr  y}  \prod_j   t_{jj}^{-j} q_{jj}^{-j} \, dt \, dq
%\\
&=  \int\int p_s( t^T t) p_s(q^Tq) \det (q^T q)^\mu e^{-\frac12  \Tr q^Tq-\frac12 \Tr  t^T t}  \prod_j   t_{jj}^{-j}  q_{jj}^{-j} \, dt dq \nonumber
\\
& = \int \int \prod_{j=1}^r t_{jj}^{r_j-j }  q_{jj}^{r_j+2\mu - j}e^{-\frac12 \sum_{i\le j} (q_{ij}^2+t_{ij}^2)}  \, dt dq.  
\end{align}
In going from line one to line two, we replaced $y = t^Tt $ with $q^T y q$, used the first invariance in \eqref{InvMeasure2}, and then also the fact  
$ p_s(q^T y q)=p_s(y) p_s(q^Tq)$, see  Proposition 1 of \cite[\S 4.2.1]{Terras}. The final line can be computed explicitly  (each diagonal component
producing a gamma function, the Gaussian integral over each off diagonals producing a factor of $2\pi$), but for what we do here it is better to leave the answer in this form.

\medskip

{\bf{Step 2: }} Denote $V_1(y) = \Tr(y) V(y)$. By comparison to the last step we easily find that
\begin{equation}
\label{V1hat}
    2^{-2r} \hat{V}_1(s) = \int \int  \Tr(t^T t q q^T) \, \prod_{j=1}^r  t_{jj}^{r_j-j }  q_{jj}^{r_j+2\mu - j}e^{-\frac12 \sum_{i\le j} (q_{ij}^2+t_{ij}^2)} dq dt. 
\end{equation}
Now expanding out,
$
\Tr( t^T t q q^T ) =\sum_{ k \le i,j \le \ell } t_{k i} t_{k j} q_{i \ell} q_{j \ell},
$
we see that any term with $i < j$ will vanish by producing a factor of the form $\int u e^{-\tfrac12 u^2} du =0$. Hence, we can replace the trace
inside the integral \eqref{V1hat} with
$ \sum_{ k \le i \le \ell } t_{ki}^2  q_{i \ell}^2.$ Each integral in the resulting sum (over $k \le i \le \ell$) can be reduced to \eqref{Mellin1} after an integration by parts (or two),  yielding a different multiplicative factor of $\hat{V}$ for different pairings of indices ($k< i < \ell$ versus $k=i=\ell$, for example). Those multipliers are as follows.
$$
  \mbox{Multiplier: } \quad \begin{array}{l} 1, \\  r_i + 2 \mu - i +1,  \\ r - i +1,  \\ (r_i -i +1) ( r_i + 2 \mu  - i +1), \end{array} 
  \quad \mbox{ for } \quad  \begin{array}{l} k < i < \ell, \\ k < i = \ell, \\ k=i < \ell,   \\ k = i = \ell. \end{array}
$$
Summing up we find a total multiplier of 
\begin{equation}
\label{first_multiplier}
 c_1 =  \sum_{i=1}^r r_i^2 + (2 \mu + r+1) \sum_{i=1}^r r_i - 2 \sum_{i=1}^r i r_i,  
\end{equation}
that is, $\hat{V_1} = c_1 \hat{V}$.

\medskip

{\bf{Step 3}: } Now let $\hat{V_2}(s) =  \int p_s(y) \int (\det x)^{-\mu }\Tr(x y) e^{-\frac{1}{2}  \Tr (  y x  + x^{-1} )} d\mu(x) d\mu(y)$, for which we have that
\begin{equation}
\label{V2hat}
    2^{-2r} \hat{V}_2(s) = \int \int  \Tr(t^T t ) \, \prod_{j=1}^r  t_{jj}^{r_j-j }  q_{jj}^{r_j+2\mu - j}e^{-\frac12 \sum_{i\le j} (q_{ij}^2+t_{ij}^2)} dq dt. 
\end{equation}
With $\Tr  ( t^T t ) = \sum_{k \le i, k \le j} t_{ki} t_{k j}$, the considerations are even simpler than above. Comparing \eqref{V2hat} to \eqref{Mellin1}, there are just two different cases.
$$
  \mbox{Multiplier: } \quad \begin{array}{l} 1,  \\ r - i +1, \end{array} 
  \quad \mbox{ for } \quad  \begin{array}{l} i < j, \\ i = j. \end{array}
$$
Summing over all possible $i, j$ we  find that  $\hat{V_2} = c_2 \hat{V}$ with
\begin{equation}
\label{second_multiplier}
  c_2 =  \sum_{i=1}^r r_i. 
\end{equation}

\medskip

{\bf{Step 4:} } Finally set $\hat{V_3}(s) =  \int p_s(y) \int (\det x)^{-\mu }\Tr(x y x y ) e^{-\frac{1}{2}  \Tr (  y x  +  x^{-1} )} d\mu(x) d\mu(y)$, and write
\begin{equation}
\label{V3hat}
    2^{-2r} \hat{V}_3(s) = \int \int  \Tr(t^T t t^T t) \, \prod_{j=1}^r  t_{jj}^{r_j-j }  q_{jj}^{r_j+2\mu - j}e^{-\frac12 \sum_{i\le j} (q_{ij}^2+t_{ij}^2)} dq dt. 
\end{equation}
Now the expansion is
\[
\Tr(t^T t t^T t)=\sum_{ k,\ell \le i, j } t_{k i} t_{k j} t_{\ell i} t_{\ell j}, 
\]
and similar to the $ \hat{V_1}$ calculation, all terms corresponding to $ k \neq \ell < i \neq j$ terms will vanish.  The six remaining choices yield:
$$
  \mbox{Multiplier: } \quad \begin{array}{l} 2 \times 1 ,  \\ 2\times( r - i +1), \\ 3, \\  (r_i -i +3)(r_i-i +1), \end{array} 
  \quad \mbox{ for } \quad  \begin{array}{l} k =  \ell < i < j  \mbox{ or } k < \ell < i =j , \\  k = \ell = i < j \mbox{ or } k < \ell = i =j,  \\
       k = \ell < i = j, \\ k = \ell = i = j . \end{array}
$$
The additional factor of two in lines one and two count the ordering of $(i, j)$ or $(k,\ell)$.  A little algebra shows that then
$ \hat{V_3}  = c_3 \hat{V}$ with
\begin{equation}
\label{third_multiplier}
  c_3  =  \sum_{i=1}^r r_i^2 + 2 (r+1)\sum_{i=1}^r r_i  - 2 \sum_{i=1}^r i  r_i.
\end{equation}
The proof of Theorem \ref{thm:MatrixBessel} is finished by checking that
$$
    c_3 + (2 \mu - r-1) c_2 - c_1 = 0,
$$
recall \eqref{first_multiplier} and  \eqref{second_multiplier}.

\medskip

To finish the proof of Theorem  \ref{thm:MatrixBessel}, we now return to the:

\begin{proof}[Proof of Lemma \ref{normlemma}]
This calculation can basically be found in \cite{Norris}. It would be nice to have a way to control at least the matrix norm as sharply without going to eigenvalue coordinates. 

Staying in the setting just considered, we let $y_1 \le y_2  \le \cdots \le y_r$ be the eigenvalues of $Y_t$ and show that
$ \lim_{t \rightarrow \infty} \frac{1}{t} \log y_i(t) =  - 2\mu +  i -1$ with probability one for each $i = 1, \dots, r$.

With $\gamma_i = \log y_i$ we find from \eqref{YIto} and considerations similar to those behind Proposition \ref{Prop:commoneigs} that,
$$
   d \gamma_i = 2 d b_i - 2 \mu dt + \sum_{j \neq i} \frac{ e^{\gamma_i} + e^{\gamma_j} }{ e^{\gamma_i} - e^{\gamma_j}} dt.
$$
One checks that  $  \sum_{j \neq r} \frac{ e^{\gamma_r} + e^{\gamma_j} }{ e^{\gamma_r} - e^{\gamma_j}}  \ge r-1$ and
$  \sum_{j \neq 1} \frac{ e^{\gamma_1} + e^{\gamma_j} }{ e^{\gamma_1} - e^{\gamma_j}}  \le  1-r$. Moreover, if we change $i$ to $i+1$ then the interaction term will change by  at most  $2  \frac{ e^{\gamma_{i+1}} + e^{\gamma_i} }{ e^{\gamma_{i+1}} - e^{\gamma_i}}.$  Thus, $\gamma_{i+1} - \gamma_i$ is bounded above by the solution to
$$
  d  z_i = 2 d (b_{i+1} - b_{i}) + 2  \left(
  \frac{1 + e^{-z_i}}{1 - e^{-z_i}} \right) dt. 
$$
The proof is finished by remarking that $ \pr ( \lim_{t \rightarrow \infty} \frac{z_i}{t} = 2, \, i =1, \dots r-1 ) = 1$.  
\end{proof}

\section{Process level identities}
\label{sec:ProcessDufresne}

We prove Theorem \ref{thm:ProcessDufresne} and the  Burke property statements of  Theorems \ref{thm:OIOO} and \ref{thm:TITO}.

\subsection{$A_t^{(\mu)}$ and $A_t^{(-\mu)}$}

Here we again start by taking $M_t=M_t^{(-\mu)}$ defined by \eqref{Magain} for $t\ge 0$ with  driving matrix Brownian motion $t \mapsto B_t$. Throughout this section we drop the superscript on the corresponding additive functional, $A_t = A_t^{(-\mu)}$, and it is always assumed that $2 \mu > r-1$.

We need the following two facts.

\begin{proposition} 
\label{claim:enlarge}
Denote by $\mathcal{B}_t = \sigma( B_s , s \le t )$ and by $\hat{\mathcal{B}}_t$ the initial enlargement $\mathcal{B}_t \vee \sigma(A_{\infty})$.
Then,
\begin{align}
\label{BMhat}
\hat B_t&:=B_t-\int_0^t \left(2\mu I-M_s^T (A_\infty-A_s)^{-1} M_s\right) ds, 
%&=B_t-2\mu t I+\int_0^t  M_s^T (Q_\infty-Q_s)^{-1} M_s ds \nonumber
\end{align}
is a standard matrix Brownian motion with respect to $\hat{\mathcal{B}}_t$ and 
 is independent of $A_\infty$. 
%In particular if we condition on $Q_\infty=\Xi$ then 
%\[
%\tilde B_t=B_t-2\mu t I+\int_0^t  M_s^T (\Xi-Q_s)^{-1} M_s ds 
%\]
%is a standard matrix BM. 
\end{proposition}

\begin{proposition} 
\label{claim:Inverses}
Almost surely, 
\begin{equation}
\label{eq:Inverses}
   (A_t^{-1}-A_{\infty}^{-1})^{-1}=\int_0^t N_s^{} N_s^T ds, \quad  \mbox{ for } N_t := A_\infty(A_\infty-A_t)^{-1} M_t. 
\end{equation}
Furthermore, conditioned on the value of $A_{\infty}$, the process $N_t$ satisfies
\begin{equation}
\label{Heq}
dN_t=N_t d\hat B_t +  (\frac12 + \mu) N_t  dt , \qquad N_0=I,
\end{equation}
where $\hat B_t$ is as defined in \eqref{BMhat}.
\end{proposition}

Granted these propositions  the result is immediate:

\begin{proof}[Proof of Theorem \ref{thm:ProcessDufresne}]
By \eqref{Heq}, $ \int_0^t N_s^{} N_s^T ds$ has the same distribution as $A_t^{(\mu)}$ as a process. 
Thus  conditioned on $A_{\infty}$,  we have that
$ \{ A_t^{-1}-A_{\infty}^{-1}, \, t \ge 0 \}  $  is equal in law to  $\{  (A_t^{(\mu)})^{-1}, \,  t \ge 0 \}$. But then this is also true unconditionally, which is precisely the desired  statement. Afterward, one can read \eqref{eq:Inverses} as an almost sure version of the identity \eqref{eq:ProcessDufresne}.
 \end{proof}

 \begin{proof}[Proof of Proposition \ref{claim:enlarge}]
 We start with the representation
\[
A_\infty = A_t + M_t   \left( \int_0^{\infty} M_t^{-1} M_{s+t}^{}  M_{s+t}^T  M_t^{-T} ds  \right)  M_t^{T},
\]
noting that by the matrix Dufresne identity (Theorem \ref{thm:MatrixDufresne}) the integral on the right hand side has the
 $\gamma^{-1}_{ 2\mu}$ distribution.
Then, with $f$ denoting the corresponding density function,
\begin{align}
\label{chanvar}
\pr ( A_{\infty} \in dA \,| \, \cB_t) = f  \left( M_t^{-1} (A - A_t^{} ) M_t^{-T}   \right) \left(\det M_t \right)^{-{(r+1)}} dA.
\end{align}
Here we have used that the Jacobian of the map $A \mapsto M A M^T$ on symmetric matrices is given by $(\det M)^{r+1}$, see for example Lemma 2.2 of \cite{Letac}.

Next, by the tower property of conditional expectations the Radon-Nikodym derivative
\begin{align}
\label{Radon}
R_{t, A}&=\frac{dP ( A_{\infty} \in dA  \, | \, \cB_t)}{ dP ( A_{\infty} \in dA)} \\
&= c_A \det(A-A_t)^{- \mu  - \frac{r+1}{2}} \det (M_t)^{2\mu} \exp[-\frac12 \Tr( (A- A_t^{})^{-1} M_t^{} M_t^T) ] \nonumber
\end{align}
is a $\cB_t$-martingale. The formula in line two follows from writing out the appearance  of $f( M_t^{-1} (A - A_t^{} ) M_t^{-T}  )$ in \eqref{chanvar}. The factor $c_A$ is  then the  density $f(A)$ without the normalizing constant. As noted in \eqref{detM00}, $t \mapsto \det M_t^{(-\mu)}$ is the geometric Brownian motion  $\det M_t = e^{ \Tr B_t -\mu r t}$. Substituting this in the preceding formula we record
\begin{align}
\label{Requation}
d R_{t,A} = R_{t,A} \left( 2\mu \Tr dB_t^{}- \Tr [ dB_t M_t^T (A-A_t)^{-1} M_t^{}]  \right)
\end{align}
for later use.

%From here everything works as in the one dimensional case.
Next we will show that for any  bounded continuous test function $h$ and 
event $\Lambda_s\in \cB_s$, $s<t$:
\begin{align}\label{condexp}
\ev \left[ \ind_{\Lambda_s} h(A_\infty) (B_t-B_s)\right]=\ev \left[\ind_{\Lambda_s}   h(A_\infty) \int_s^t  
\left(2\mu-M_u^T(A_\infty-A_u)^{-1}A_u\right) du\right].
\end{align}
Granted \eqref{condexp} the monotone class theorem will imply that
\[
\ev \left[ B_t-B_s\big|  \hat\cB_s \right] =\int_s^t  \left(2\mu-M_u^T(A_\infty-A_u)^{-1}M_u\right) du,
\]
or in other words that  $\hat B_t$ defined in \eqref{BMhat} is a local martingale (with respect to the filtration $\hat\cB_t$). 
To see that it is actually a matrix Brownian motion, and so complete the proof, one checks the quadratic covariation of 
the entries and invokes  L\'evy's theorem.  

Returning now to (\ref{condexp}) we introduce $\lambda_t(h):= \ev [ h(A_\infty) | \cB_t]$ and write the left hand 
side of that equality as in 
\[
\ev \left[ \ind_{\Lambda_s} h(A_\infty ) (B_t-B_s)\right]= \ev \left[ \ind_{\Lambda_s} (\lambda_t(h)B_t-\lambda_s(h) B_s)\right],
\]
as follows by conditioning separately with respect to both $\cB_t$ and then $\cB_s$. 
By \eqref{Radon} we have that,
\begin{align*}
\lambda_t(h) &=\int_\mathcal{P} h(A) f (M_t^{-1} (A-A_t) [M_t]^T) (\det M_t)^{-2} dA\\
&=\int_\mathcal{P} h(A) R_{t, A} f(A) dA,
\end{align*}
where again $f$ is the $\gamma_{r, 2\mu}^{-1}$ density function. And then \eqref{Requation} implies that,
\[
\lambda_t(h)-\lambda_s(h) =\int_\mathcal{P} \int_s^t h(A) R_{u,A} \left( 2 \mu \Tr dB_u- \Tr [ dB_u M_u^T (A-A_u)^{-1} M_u]  \right) f(A) dA.
\]
To continue we compute the cross variation of $\lambda_t(h)$ and $[B_t]_{i,j}$,  for which
 we just need to check the coefficient of $dB_{i,j}$ in the previous integral with the result that
\[
\langle \lambda_t(h) , [B_t]_{ij}\rangle = 
\int_\mathcal{P} h(A) R_{A,u} \left( 2\mu \ind_{i=j}- \left[M_u^T (A-A_u)^{-1} M_u\right]_{i,j}  \right) f(A) dA.
\] 
Then writing $B_t - B_s = \int_s^T d B_u$  leads to 
\begin{align*}
\ev \left[  (\lambda_t( h)B_t-\lambda_s(h) B_s)\big| \cB_s \right]
&=\int_\mathcal{P} \int_s^t h(A) R_{A,u} \left(2\mu I -M_u^T(A-A_u)^{-1}M_u\right) du \, f(A) dA \\
&= \int_s^t \ev \left[ h(A_{\infty})  \left(2\mu I -M_u^T(A_{\infty}-A_u)^{-1}M_u\right) \big| \cB_u\right]du .
\end{align*}
From here we see that $\ev \left[ \ind_{\Lambda_s} (\lambda_t(h)B_t-\lambda_s(h) B_s)\right]= \ev \left[\ind_{\Lambda_s} \ev\left[  (\lambda_t(h)B_t-\lambda_s(h) B_s)\big| \cB_s\right]\right]$ is equal to the right hand side of (\ref{condexp}), as required.
 \end{proof}

 \begin{proof}[Proof of Proposition \ref{claim:Inverses}]
 Since $t \mapsto A_t^{-1}$ is almost surely once differentiable and $|| A_t - t I || = o(t) $ as $t \rightarrow 0$ (as $M_0=I$ and $M_t$ is continuous), we have that
\begin{equation}
\label{eq:Inverses1}
(A_t^{-1}-A_{\infty}^{-1})^{-1}=\int_0^t d\left[(A_s^{-1}-A_{\infty})^{-1}\right],
\end{equation}
also almost surely.
On the other hand,
\begin{align*}
d\left[(A_t^{-1}-A_{\infty}^{-1})^{-1}\right]&=dA_t (A_\infty-A_t)^{-1} A_\infty+A_t(A_\infty-A_t)^{-1}dA_t(A_\infty-A_t)^{-1}A_\infty\\
&=M_tM_t^T  (A_\infty-A_t)^{-1} A_\infty dt+A_t(A_\infty-A_t)^{-1}M_tM_t^T(A_\infty-A_t)^{-1}A_\infty dt\\
&=A_\infty (A_\infty-A_t)^{-1}M_t M_t^T(A_\infty-A_t)^{-1}A_\infty dt.
%=H_tH_t^T dt
\end{align*}
Here we have used the matrix identity $(C^{-1}-D^{-1})^{-1}=C(D-C)^{-1} D$ in line one, and 
 $I+C(D-C)^{-1}= D(D-C)^{-1}$ to go from line two to line three. 
 Substituting back into \eqref{eq:Inverses1} yields \eqref{eq:Inverses}, identifying $N_t =  A_\infty (A_\infty-A_t)^{-1}M_t $ at the same time.
 
 To see \eqref{Heq} note that $N_0 = A_{\infty}^{} A_{\infty}^{-1} M_0 = I$ and compute,
 \begin{align*}
d  &  [ A_\infty (A_\infty-A_t)^{-1}    M_t ] \\
& =  \,  A_\infty(A_\infty-A_t)^{-1} M_t M_t^T(A_\infty-A_t)^{-1}M_t dt + A_\infty (A_\infty-A_t)^{-1}M_t(dB_t+(\frac12-\mu) I dt)\\
& =  \,  N_t\left[ M_t^T(A_\infty-A_t)^{-1}M_tdt +d B_t +(\frac12-\mu)I dt  \right].  
\end{align*}
By Proposition \ref{claim:enlarge}, the quantity inside the final bracket equals $ d\hat{B}_t + (\frac12 + \mu) I dt $, as desired. \end{proof}

\subsection{Burke properties}

Both results are consequences of Propositions \ref{claim:enlarge} and \ref{claim:Inverses}. Let $N_t$ be defined as in the latter statement and introduce the temporary shorthand $\hat{A}_t = \int_{0}^t N_s N_s^T ds$, reserving  
$ A_t $ for $\int_0^t M_s M_s^T ds $. Importantly,  $M_t$ is also chosen as throughout the proofs of those same propositions. In particular
the underlying matrix Brownian motions $\hat{B}_t $ and $B_t$ stand in the same relationship.

\begin{proof}[Proof of Theorem \ref{thm:OIOO}]
 We start by rewriting the  almost sure identity \eqref{eq:Inverses} as in
  $$
        \hat{A}_t^{-1}  =  A_t^{-1} - A_{\infty}^{-1}. 
  $$
  Repeated use of this along with the resolvent identity then  produces
 \begin{align}
 \label{algebraparty}
A_\infty (A_\infty-A_t)^{-1} & = (I-(\hat A_t^{-1}+A_\infty^{-1})^{-1} A_\infty^{-1})^{-1}\\
&= (I-( I + A_\infty \hat A_t^{-1} )^{-1} )^{-1} \nonumber \\
&= \hat A_t A_\infty^{-1}(I + A_\infty \hat A_t^{-1} ) \nonumber  \\
&=(A_\infty+ \hat A_t ) A_\infty^{-1}. \nonumber
\end{align}
The important observations are:  the left hand side of \eqref{algebraparty} times $M_t$ (from the right) is the definition of $N_t$, and, on the right hand side, we have that $A_{\infty}$ is independent of $\hat{A_t}$ (on account of being independent of $\hat B_t$ and so $N_t$). 

To exploit the second point, we extend $N_t$ to $t \in (-\infty, 0)$ as in Section  \ref{sec:DufresneBasic}. 
%Simply extend the $\hat B_t$ to the negative half-line by a new (one sided) matrix Brownian motion, independent of both $B_t $ and $\hat B_t$ for $t \ge 0$, and let the new  process be defined by the solution of \eqref{Heq} over the whole line. 
By the same reasoning used in the proof of Theorem \ref{thm:MatrixDufresne}, 
 $\int_{-\infty}^0 N_s^{} N_s^T ds $ is an independent copy of $A_{\infty}$. %(as we learned ).

Now, modifying the introduced notation to let $\hat A_{(-\infty, t)} $ denote $\int_{-\infty}^t N_s^{} N_s^T ds$, we conclude from the above comments and \eqref{algebraparty} that
$$
    \hat A_{(-\infty, 0)} (\hat A_{(-\infty, t)})^{-1}  N_t   \eqd  M_t,
$$
as processes for $t\ge 0$. In order to recognize this as equivalent to the identity \eqref{eq:MatrixOIOO} announced in Theorem \ref{thm:OIOO}, take inverse-transposes throughout the above to find that,
\begin{equation}
\label{eq:MatrixOIOO1}
      (\hat A_{(-\infty, 0)})^{-1}  N_t ( N_t^{-1}  \hat A_{(-\infty, t)}  N_t^{-T} )   \eqd  M_t^{-T}.
\end{equation}
An application of It\^o's formula will then show that $M_t^{-1}$ is a copy of $N_t$ . In particular, it satisfies $d M_t^{-1} = M_t^{-1}  d (- B_t^T) 
+ (\frac12+\mu) M_t^{-1} dt$, for $t \ge 0$.

Finally, the independence statement follows from the same trick used in  \cite{OconYor}  for the one-dimensional case (see the proof of Theorem 6 there). Bringing in yet more notation, let $L_t$ denote the left hand side of \eqref{eq:MatrixOIOO1}. We will show that
\begin{equation}
\label{AnotherTimeReversalTrick}
    N_t^{-1}  \hat A_{(-\infty, t)}  N_t^{-T} = L_t^T  ( \int_t^{\infty} L_s^{-T} L_s^{-1} ds )  L_t,
\end{equation}
with probability one. The independence of $\{ L_s, s \le t \}$ and $\{  N_s^{-1}  \hat A_{(-\infty, s)}  N_s^{-T}, s > t\}$   being made clear by writing the right hand side of \eqref{AnotherTimeReversalTrick} as $\int_t^{\infty} (L_t^{-1} L_s)^{-T}  (L_t^{-1} L_s)^{-1} ds$. On account of
\eqref{eq:MatrixOIOO1} the process $L_t$ inherits the independence of multiplicative increments from $M_t$.
To verify \eqref{AnotherTimeReversalTrick} notice that: 
\begin{align*}
L_t^T  ( \int_t^{\infty} L_s^{-T} L_s^{-1} ds )  L_t^{}  & = N_t^{-1}  \hat A_{(-\infty, t)}  \left(  \int_t^{\infty} \hat A_{(-\infty, s)}^{-1} N_s N_s^T \hat A_{(-\infty, s)}^{-1}  ds  \right) 
    \hat A_{(-\infty, t)}  N_t^{-T} \\
                              & = N_t^{-1} \hat A_{(-\infty, t)}   \left( - \int_t^{\infty}    d ( \hat A_{(-\infty, s)}^{-1} )   \right)  \hat A_{(-\infty, t)}  N_t^{-T}.
\end{align*}
The proof is finished upon integrating and using  the fact that $  || \hat A_{(-\infty, t)}^{-1}   || \rightarrow  0 $ as $t \rightarrow \infty$ with probability one (which follows by the computation behind  Lemma \ref{normlemma}).
\end{proof}

%\benedek{TITO}

As for  Theorem  \ref{thm:TITO}, the identity (\ref{algebraparty}) together with the definition of $N_t$ from (\ref{eq:Inverses}) gives that
\[
M_s^T(A_\infty-A_s)^{-1} M_s=N_s^T(A_\infty+\hat A_s)^{-1}N_s=N_s^T  \hat A_{(-\infty, s)}^{-1} N_s,
\]
where we continue using the notation introduced in the previous proof, and  hence Proposition \ref{claim:enlarge} can be rewritten as in:
\[
B_t=\hat B_t+\int_0^t (2\mu I-N_s^T  \hat A_{(-\infty, s)}^{-1} N_s) ds. 
\]
Since the right hand side only depends on $\hat B_t$, this identity provides a nonlinear transformation producing one matrix  Brownian motion from another. Reversing the roles of the Brownian motions and reverting to our original notation yields:

\begin{corollary}\label{anotherOIOO}
%Let $B_t$ be  a two-sided matrix  Brownian motion, and 
Let $2\mu>r-1$ and  now take $M_t=M_t^{(\mu)}$,  
 extended to $t\in \R$ as described in Section \ref{sec:DufresneBasic}.
 Denote the (two-sided) driving matrix Brownian motion for $M_t$ by $B_t$.
  Then 
\[
\hat B_t=B_t+2\mu I t-\int_0^t M_s^T (\int_{-\infty}^s M_u^{} M_u^T du)^{-1} M_s ds\eqd B_t 
\]
as processes for $t\ge 0$. 
\end{corollary}

From this the proof of Theorem \ref{thm:TITO} is straightforward. 

%\begin{proposition}
%Suppose that $B_t, C_t$ are independent two-sided matrix valued Brownian motions with iid standard entries and $2\mu>r-1$. Consider the drifted matrix valued Brownian motion $H_t$ corresponding to $B_t+C_t$. This is defined as the strong solution of 
%\[
%dH_t=H_t(dB_t+dC_t)+(2\mu+1) H_t dt, \qquad H_0=I,
%\] 
%for $t\ge 0$, extended to negative $t$ values using the independence stationary increment property as described in Section \ref{sec:DufresneBasic}. 
%Now define the processes
%\begin{align*}
%F_t&=B_t+2\mu I t-2 \int_0^t H_s^T \left(\int_{-\infty}^s H_u H_u^T du\right) H_s ds,\\
%G_t&=C_t+2\mu I t-2 \int_0^t H_s^T \left(\int_{-\infty}^s H_u H_u^T du\right) H_s ds.
%\end{align*}
%Then $(F_t, G_t, t\ge 0)=(B_t,C_t,\ge 0)$. 
%\end{proposition}
\begin{proof}[Proof of Theorem \ref{thm:TITO}]
The introduced process $H_{t/2}$ can be equated with $M_t$, though driven by the standard matrix Brownian motion $B_{t/2}+C_{t/2}$. Then by  the corollary we have that
\begin{align*}
\hat B_t&=B_{t/2}+C_{t/2}+2\mu I t-\int_0^t M_s^T (\int_{-\infty}^s M_uM_u^T du)^{-1} M_s ds\\
&=B_{t/2}+C_{t/2}+2\mu I t- \int_0^{t/2} H_s^T (\int_{-\infty}^s H_u H_u^T du)^{-1} H_s ds
\end{align*}
is another standard matrix Brownian motion.
Checking the definitions we see that  $F_t=\frac{\hat B_{2t}}{2}+\frac{B_t-C_t}{2}$ and $G_t=\frac{\hat B_{2t}}{2}-\frac{B_t-C_t}{2}$. Since $\hat B_t$ is constructed from $B_t+C_t$, it is independent (as a process) from the process $B_t-C_t$. But this means that $\{F_t, t\ge 0\}$ and $\{G_t, t\ge 0\}$ are independent of each other and they are both standard  matrix Brownian motions. 
\end{proof}

\section{The matrix $X_t$ and $Z_t$ processes}
\label{sec:XandZ}

The process $X_t$ is actually equivalent to the $Q_t$ encountered in the proof of Theorem \ref{thm:MatrixDufresne}, and hence its sde has already been recorded in \eqref{basic_Qprocess}.
As for $Z_t$, we will again rely in part on the technology developed in the last section.  Recall $M_t$  and $N_t$ from Propositions  \ref{claim:enlarge} and \ref{claim:Inverses} and define two versions of the $Z_t$ process:
\begin{equation}
\label{twoZs}
   Z_t = M_t^{-1} \int_0^t M_s^{} M_s^{T} ds, \qquad \hat{Z}_t = N_t^{-1}  \int_0^t N_s^{} N_s^{T} ds.
\end{equation}
That is, $Z_t$ corresponds to $-\mu$ and is driven by $B_t$, $\hat{Z}_t$ to $+\mu$ and $\hat{B}_t$, and $B_t$ and $\hat B_t$ are related by Proposition \ref{claim:enlarge}.\footnote{Note that in the discussion of Section \ref{sub:IntroXandZ}, $Z_t$ was defined through $M_t$ with the  $+\mu$ drift term. The  choice to flip things around here is natural given the course of the previous arguments.}
  It\^o's formula yields
\begin{equation}
\label{Zprocesses}
  dZ_t = (\frac12 + \mu) Z_t  dt + M_t^T dt - dB_t Z_t, \quad    d \hat{Z}_t = (\frac12 - \mu) \hat{Z}_t  dt + N_t^T dt - d \hat B_t \hat Z_t.
\end{equation}
We first show how to close these equations using properties of the matrix GIG distribution (proving Theorem \ref{thm:MatrixXandZ} 
and then Corollary \ref{cor:Intertwining}). After that we consider the asymptotics of the underling eigenvalue processes 
(proving Theorem \ref{thm:Laplace}). 

\subsection{The role of the GIG}

The following rather surprising fact already 
implies the invariance in law of $Z_t$ under the map $\mu \mapsto - \mu$.

\begin{proposition}
\label{ZequalsZhat}
$Z_t = \hat Z_t$ almost surely for  $t\ge0$. 
\end{proposition}

\begin{proof}
From Proposition \ref{claim:Inverses} we have that $ \hat Z_t= N_t^{-1}  \int_0^t N_s N_s^T ds =
N_t^{-1} (A_t^{-1}-A_{\infty}^{-1})^{-1} $ as  
well as $N_t=A_\infty(A_\infty-A_t)^{-1} M_t$. Hence,
\begin{align*}
\hat Z_t & = (A_\infty(A_\infty-A_t)^{-1} M_t)^{-1} (A_t^{-1}-A_{\infty}^{-1})^{-1}\\
& = M_t^{-1} (A_\infty-A_t)A_\infty^{-1}  (A_t^{-1}-A_{\infty}^{-1})^{-1} A_t^{-1} A_t^{} \\
& = M_t^{-1}  A_t^{}, 
\end{align*}
which is the definition of $Z_t$. 
\end{proof}

For the Markov property we need:

\begin{proposition}
\label{prop:GIG}
The conditional distribution of $M_t^T A_t^{-1} M_t^{}$
%$M_t^T Z_t^{-1}$
 given $\{ Z_s, s \le t, Z_t = Z \}$ is the matrix GIG law $\eta_{-\mu,I, (ZZ^T)^{-1}}$. The conditional distribution of $Z_t(A_t^{-1}-A_{\infty}^{-1})Z_t^T$ given the same $\sigma$-field  is $\eta_{-\mu, (ZZ^T)^{-1} ,I}$.
\end{proposition}

We go ahead with the consequences of Proposition \ref{prop:GIG}, and return to its proof at the end of the section.

\begin{proof}[Proof of Theorem \ref{thm:MatrixXandZ}]
We follow the strategy of \cite{MY3}, employing Theorem 7.12 of \cite{LS} to close the equation for $Z_t$ in \eqref{Zprocesses}
as follows. Taking conditional expectation throughout with respect to $\mathcal{Z}_t = \sigma( Z_s, s \le t )$, the ideas there show that 
\begin{equation}
\label{prelimZ_1}
  dZ_t = (\frac12 + \mu) Z_t  dt + \ev [  M_t^T  \, | \, Z_s, s \le t ]  dt - d W_t Z_t,
\end{equation}
where $W_t$ is matrix Brownian motion adapted to $\mathcal{Z}_t \subsetneq \mathcal{B}_t = \sigma(B_s, s \le t)$. Next, 
since
$M_t^T A_t^{-1} M_t^{T} = M_t^T Z_t^{-1} \sim \eta_{-\mu, I, (Z_t^{}Z_t^T)^{-1}}$ by Proposition \ref{prop:GIG}, there is the formula
$\ev  [  M_t^T  \, | \, Z_s, s \le t ]  = \kappa_{-\mu}( I, (Z_t Z_t^T)^{-1}) Z_t$. We can then rewrite the above as in 
\begin{equation}
\label{prelimZ}
     dZ_t = (\frac12 + \mu) Z_t  dt +  \kappa_{-\mu}( I, (Z_t Z_t^T)^{-1}) Z_t   dt - d W_t Z_t,
\end{equation}
and believing that one can let $\mu \mapsto -\mu$,  $Z_t$ is the stated diffusion (in its own filtration).

Note that above construction entails that  
$$
W_t = B_t+\int_0^t [M_s^T Z_s^{-1}- (\frac12 + \mu) - \kappa_{-\mu}( I, (Z_s Z_s^T)^{-1}) ] ds.
$$
Similar to the case considered in \cite{LS}, one checks that this is a matrix Brownian motion by using of It\^o's formula to find an ordinary differential equation
for $t  \mapsto \ev [ e^{ i \Tr (C W_t)} ]$, $C \in GL_r$.

To verify that the formulas themselves hold up under the change of sign, consider $\hat Z_t$ for which the above procedure leads to the following analog of
\eqref{prelimZ_1}:
\begin{equation}
\label{prelimhatZ_1}
d\hat Z_t = ( \frac12 - \mu) \hat Z_t dt + \ev [ N_t^T  |  \, \hat{Z}_s^{}, s \le t ]  dt - d \hat{W}_t \hat{Z}_t  dt,
\end{equation}
with a similar considerations for the new matrix Brownian motion $\hat{W}_t$. Proposition \ref{prop:GIG} again applies after noting that 
\begin{align*}
N_t^T = M_t^T  (A_\infty-A_t)^{-1} A_\infty = M_t^T  A_t^{-1} A_t (A_\infty - A_t)^{-1} A_\infty=  Z_t^{-T} 
(A_t^{-1}-A_\infty^{-1})^{-1}. 
\end{align*}
That is, $N_t^{T} Z_t^{-1} = (Z_t  (A_t^{-1} -A_{\infty}^{-1})Z_t^T)^{-1}$ which has law $\eta_{\mu, I, (Z_t {Z_t}^{T})^{-1}}$ conditional on $\{ {Z}_s, s \le t \} $. 
Then by Proposition \ref{ZequalsZhat} we also have that $N_t^{T} \hat Z_t^{-1} $ has this same law conditional on $\{ \hat{Z}_s, s \le t \} $. Substituting into \eqref{prelimhatZ_1} gives the desired sde, that is, \eqref{prelimZ} with a sign flip on the parameter $\mu$.
\end{proof}

\begin{proof}[Proof of Corollary \ref{cor:Intertwining}]

Proposition 2.1 of \cite{CPY} provides a soft criteria for two processes $\mathsf{X}_t$ and $\mathsf{Y}_t$ defined on the same probability space  to intertwine. With $\mathsf{X}_t$ taking values in $E$ and $\mathsf{Y}_t$ taking values in $F$ (possibly separate measure spaces), it is assumed that:

\medskip

 (i) $\textsf{X}_t$ is Markovian with respect to a filtration $\mathcal{F}_t$,  and $\textsf{Y}_t$ is Markovian with respect to a filtration $\mathcal{G}_t$ such that $\mathcal{G}_t \subset \mathcal{F}_t$,

 (ii)  There exists a Markov kernel $\Lambda: E \mapsto F$ for which  
$
   \E [ h(\textsf{X}_t) | \mathcal{G}_t]  = ( \Lambda h) (\textsf{Y}_t)
$
for all $t>0$ and  integrable  $h: E \mapsto \mathbb{R}_+$.  

\medskip

\noindent
Given this the outcome is that
$T^{\textsf{Y}}_t \Lambda = \Lambda T^{\textsf{X}}_t$ as operators under additional ``mild continuity assumptions".

From what we have just shown the above applies directly  to $ \mathsf{X}_t = N_t$ and $\mathsf{Y}_t = \hat{Z}_t$. The mild continuity assumptions being easily satisfied as both choices are continuous pathed Feller processes on $\R^{r\times r}$. 
(Note since the original statement takes the positive $\mu$ drift, it is consistent to consider here $(N_t, \hat{Z}_t)$ rather than $(M_t, Z_t)$ for which there is a corresponding result.)

For (i): $N_t$ is Markovian with respect to $\hat{\mathcal{B}}_t = \sigma( \hat B_s, s \le t)$ and $\hat{Z}_t$ is Markovian with respect to its own filtration $\mathcal{Z}_t$ which the proof of Theorem \ref{thm:MatrixXandZ} shows is contained strictly inside $\hat{\mathcal{B}}_t$.
And (ii) is the second point of Proposition \eqref{prop:GIG}:  $\hat{Z}_t^{-T} N_t^{} \sim \eta_{\mu, I, (Z Z^T)^{-1}}$ and so,
$$
  \ev[ h (N_t) | \hat{\mathcal{Z}}_t, \hat{Z}_t = Z ] = \int_{GL_r} h(Z^T X) \eta(dX)  := \Lambda h(Z),
$$
where $\eta = \eta_{\mu, I, (Z Z^T)^{-1}}$,

For the second part, take $ \mathsf{X}_t = \hat{Z}_t^{} N_t^{-T}$. Then the representation of the intertwining kernel by
$
    \Lambda h(Z)  = \int_{\mathcal{P}}  h(X^{-1}) \eta(dX),
$
now viewed as from $\mathcal{P}$ into $GL_r$, follows from $ (N_t^T  \hat{Z}_t^{-1})^{-1}$ having conditional law $\eta_{-\mu, I (ZZ^T)^{-1}, I} = \eta^{-1}$.
 \end{proof}

We now return to the proof of Proposition  \ref{prop:GIG}.   
The first step, Lemma \ref{lem:Bernadac} below,  is designed to implement Bernadac's characterization of the matrix GIG  
\cite[Theorem 5.1]{Bernadac} to this end. The notation $\gamma_{p, A}$ introduced in the statement
refers to the non-central Wishart law, which has density proportional 
 $(\det X)^{\frac{p-r-1}{2}} e^{-\frac12 \Tr (A^{-1} X)} \ind_{\mathcal P}(X)$ for $A \in \mathcal{P}$. It is worth pointing out  that the posited independence structure in the statement (of $(X+Y)^{-1}$ and $X^{-1} - (X+Y)^{-1}$ given that of  $X$ and $Y$) is now commonly referred to as the Matsumoto-Yor property.
The proof of Proposition \ref{prop:GIG} is then completed by specifying our particular choice of $X$ and $Y$ in Lemma \ref{lem:ApplyBernadac}.

\begin{lemma} 
\label{lem:Bernadac}
Suppose that the $\mathcal{P}$-valued random variables  $X$ and $Y$ are independent, and that $(X+Y)^{-1}$ and  $X^{-1}-(X+Y)^{-1}$ are also independent. Suppose further that $Y\sim \gamma_{2p,A^{-1}}$ and $X^{-1}-(X+Y)^{-1} \sim \gamma_{2p,  B^{-1}}$ 
for $A, B \in \mathcal{P}$. Then $X\sim  \eta_{-p, A,B}$ and $(X+Y)^{-1}\sim \eta_{-p,B,A}$.
\end{lemma}

\begin{proof}
Set $U=(X+Y)^{-1}$ and $V=X^{-1}-(X+Y)^{-1}$.  Let $Y'$ be a copy of $V$ (that is, $Y' \sim  \gamma_{2p,  B^{-1}}$) independent of both  $X$ and $Y$.  Then 
\begin{align*} 
X=(U+V)^{-1}\eqd (Y'+U)^{-1}=(Y'+(Y+X)^{-1})^{-1}.
\end{align*}
Bernadac's result is that the above distributional identity holds with $X, Y, $ and $Y'$ all independent and $Y$ and $ Y'$ having the  corresponding Wishart distributions if and only if  $X \sim \eta_{-p, A,B}$.

Alternatively, if we let $V^{\prime}$ be a copy of $Y$  (so $V' \sim  \gamma_{2p, A^{-1}}$) independent of $U$ and $V$ we will have that
$$
   U = (Y + (V+ U)^{-1})^{-1}  \eqd ( V' + (V + U)^{-1})^{-1},
$$
and the result is that $U = (X+Y)^{-1} \sim \eta_{-p, B,A}$.
\end{proof}

\begin{lemma}
\label{lem:ApplyBernadac} 
Set
\[
X = M_t^T {Z_t}^{-1} = M_t^T A_t^{-1} M_t, \qquad Y = M_t^T (A_\infty - A_t) ^{-1} M_t.
\]
Then,
$$
 (X+Y)^{-1} = Z^{}_t(A_t^{-1}-A_{\infty}^{-1}) Z_t^T,
 \qquad X^{-1}-(X+Y)^{-1} = Z_t^{} A_\infty^{-1} Z_t^T.
$$
Conditioned on the $\sigma$-field generated by $\{ Z_s, s \le t \} $ the random matrices  $X$ and $Y$ are independent,
 as are  \mbox{$(X+Y)^{-1}$} and $X^{-1}-(X+Y)^{-1}$. Further,  the conditional distribution of 
$Y$ is  $\gamma_{2\mu,  A^{-1}}$ where $A=I$ and that of $X^{-1}-(X+Y)^{-1}$ is  $\gamma_{2\mu,  B^{-1}}$
where $B = (Z_t^{} Z_t^{T})^{-1}$.
\end{lemma}

\begin{proof}
The formulas for $(X+Y)^{-1}$ and $X^{-1} -(X+Y)^{-1}$ follow from some simple algebra.

Clearly $Y$ is independent of $\sigma(B_s, s\le t)$, and thus is independent of both $X$ and $\{Z_s, s \le t\}$. For the independence of $(X+Y)^{-1}$
and $X^{-1} - (X+Y)^{-1}$  we need to prove the conditional independence of $A_t^{-1}-A_{\infty}^{-1}$ and $A_\infty$.
By Proposition \ref{claim:Inverses} 
 we know that $A_t^{-1}-A_{\infty}^{-1} = \int_0^T N_s N_s^T ds  $ is measurable $\sigma(\hat B_s, s\le t)$ and so is  independent of $A_\infty$. 
 But $Z_t$ is in this $\sigma$-field as well ($Z_t = \hat Z_t$) implying there is also 
 conditional independence.
\end{proof}

\subsection{Asymptotics}

We start by identifying the underlying eigenvalue processes:

\begin{lemma}
\label{XandZspec}
Denote the (ordered, nonintersecting) eigenvalues of $X_t$ by $0 \le x_r \le x_{r-1} \le \cdots  \le x_1$. These perform the joint diffusion:
\begin{equation}
\label{Xeigs_sde}
   d x_i  = 2 x_i db_i + \left( 1+ (2-2 \mu) x_i + x_i \sum_{j \neq i} \frac{x_i + x_j}{x_i - x_j}   \right) dt. 
\end{equation}
For $Z_t$ consider instead the similarly ordered singular values $\{ z_i\}$ of  $Z_t$. This family is  also Markov, and 
is governed by
\begin{equation}
\label{Zeigs_sde}
  dz_i = z_i db_i +  \left( \big(\frac{r}{2} - \mu\big) z_i + \frac{1}{ z_i} [ \kappa_{\mu} (\Lambda_z, I) ]_{ii} + z_i \sum_{j \neq i} \frac{z_j^2}{z_i^2-z_j^2} \right) dt,
\end{equation}
where $\Lambda_z$ denotes the diagonal matrix $[ \Lambda_z ]_{ii} =  z_i^{-2}$.
\end{lemma}

\begin{proof} As $X_t$ is equivalent in law to $Q_t$ used throughout Section \ref{sec:DufresneBasic}, Proposition \ref{Prop:commoneigs}
describes the process for the inverse eigenvalues of $X_t$. That is to say that \eqref{Xeigs_sde} follows from \eqref{CommonEigSDE} upon making the substitution $x_i  = p_i^{-1}$ in the $\beta=1$ instance of that equation.

The argument for $Z_t$ requires only a couple of additional observations. Write
\begin{equation}
\label{ZZSDE}
  d (Z_t^T Z_t) = - Z_t^T ( dB_t + dB_t^T) Z_t + (1-2\mu +r) Z_t^T Z_t dt + 2 Z_t^T \kappa_{\mu}( I, (Z_t Z_t^T)^{-1}) Z_t  dt,
\end{equation}
which is now  rotation invariant. In particular, setting  $Z_t = V_t \Lambda_t^{1/2} U_t^T $ for orthogonal $U, V$ and 
$\Lambda_t$ the diagonal of square-singular values of $Z_t$, the right hand side of \eqref{ZZSDE}  equals
$$
   U_t \Bigl(  \Lambda_t^{1/2} d \mathcal{B}_t  \Lambda_t^{1/2} + (1-2\mu +r) \Lambda_t dt  + 
   2  \Lambda_t^{1/2} \kappa_{\mu}(I, \Lambda_t^{-1}) \Lambda_t^{1/2} dt \Bigr) U_t^T. 
$$
Here $\mathcal{B}_t = V_t^T ( B_t + B_t^T ) V_t $
is equal in law to twice a symmetric (or ``Dyson") Brownian motion, and we have used  that
$
    \kappa_{\mu}(I, U A U^T ) =  U {\kappa}_{\mu}(I, A) U^T,
$
for any symmetric $A$. From here the standard method used before will yield the system,
\begin{align}
\label{ZeigSDE}
  d \lambda_i  
  %& = 2 \lambda_i db_i +    [ (1+ 2\mu +r) \lambda_i + 2  \lambda_i [  {\kappa}_{-\mu}(I, \Lambda^{-1})]_{ii} 
 % + {2}\sum_{j \neq i } \frac{\lambda_i \lambda_j}{\lambda_i - \lambda_j} ] dt  \\ 
  & = 2 \lambda_i db_i +    [ (1- 2\mu +r) \lambda_i + 2  [  {\kappa}_{\mu}(\Lambda^{-1}, I)]_{ii} 
  + {2} \sum_{j \neq i } \frac{\lambda_i \lambda_j}{\lambda_i - \lambda_j} ] dt, 
\end{align}
having employed the identity $\Lambda^{1/2} \kappa_{\mu}(I, \Lambda^{-1}) \Lambda^{1/2}
   =  \kappa_{\mu}(\Lambda^{-1}, I)$ along the way.
Setting $z_i = \sqrt{\lambda_i}$ completes the proof.
\end{proof}

The proof of Theorem \ref{thm:Laplace} is now split into two parts.

\begin{proposition}\label{prop:Zasymp}
%2 \gamma is dimension discrepancy now 
Set $\mu = \frac{r-1}{2} +\gamma$ in \eqref{Xeigs_sde} and denote the maximal eigenvalue by $x_t^{\gamma}$. Then, as processes, $ \frac{1}{2c} \log x_{c^2 t}^{\gamma/c} $
converges as $c \rightarrow \infty$  to the Brownian motion with drift $-\gamma$ reflected at the origin. (The lower eigenvalues  converge to the zero process in  this scaling).
\end{proposition}

\begin{proof} Changing to logarithmic coordinates, $y_i = \log x_i$, and introducing the scaling as in $\gamma \mapsto \gamma/c $ (after putting $\mu = \frac{r-1}{2} + \gamma$)  and $y_i(t) \mapsto   y_i(c^2 t) /(2c)$, we can work with the 
system
\begin{equation}
\label{scaledXeigs} 
   d y_i = 2 db_i + \left( -  \gamma + \frac{c}{2} e^{-c y_i } +  c  \sum_{j \neq i} \frac{ e^{c y_j}}{e^{c y_i} -  e^{c y_j}}  \right) dt.
\end{equation}
Here the drift was simplified ahead of time by using  $-(r-1)+ \sum_{j \neq i} \frac{x_i + x_j}{x_i - x_j} =    \sum_{j \neq i} \frac{ 2  x_j}{x_i - x_j} $. 

On the other hand, back in the original coordinates we have that,
$$
     \pr \Bigl(  \lim_{t \rightarrow \infty} ( \log x_i(t) - \log x_j(t) ) = + \infty \Bigr) = 1,
$$
for any pair $i >j$. This follows again by the proof of Lemma \ref{normlemma}. From here we see that the top point in \eqref{scaledXeigs} shares whatever  $c \rightarrow \infty$ process limit it may have with that for 
$ y_c(t)$ defined by
\begin{equation}
\label{yc}
    {y}_c(t) =  b_t^{-\gamma}  +  L_c(t),  \qquad    L_c(t) =   \frac{c}{2}   \int_0^t e^{-c  {y}_c(s) }  ds, 
\end{equation}
and we want to show that $L_c(t)$ produces a local time contribution in the limit.

For  $\epsilon > 0$, decompose the path $t \mapsto y_c(t)$ at the time $s \le t$ at which it was last beneath level $\epsilon$ to
find that
$$
   y_c(t) \le \epsilon +  \max_{s\le t} (b_t^{-\gamma} - b_s^{-\gamma}) + \frac{c}{2} e^{-c \epsilon} t.
$$
Thus, 
$$
    L_c(t) \le  \epsilon +  \max_{s\le t} (b_t^{-\gamma} - b_s^{-\gamma}) + a_\epsilon t - b_t^{-\gamma},
$$
and  for each fixed $t$  it holds that $\sup_{c > 0} L_c(t) < \infty$ with probability one. Decomposing instead at the last time that the path 
exceeds $-\epsilon$ similar reasoning shows that $\liminf_{c \rightarrow \infty} $ $\inf_{0\le s \le t}$ $y_c(s) \ge 0$.   With both sequences bounded above and below, by passing to a subsequence if needed there exist $(y(t), L(t))$ with $y_c(t) \rightarrow y(t)$ and $L_c(t) \rightarrow L(t)$ at all but countably many $t$ for which $y(t) = b_t^{-\gamma} + L(t)$.  Since $\int_0^t \ind_{[\epsilon, \infty)} (y_c(s)) d L_c(s) \rightarrow 0$,  any limiting $L(t)$ is non-decreasing and increases only when $y(t)=0$. As $y(t) \ge 0$, we see that any such pair $(y(t), L(t))$ is the (unique) solution to the Skorohod problem for $b_t^{-\gamma}$. This precisely what we wanted to show.
\end{proof}

\begin{proposition} Now  set $\mu = \frac{r-1}{2} +\gamma$ in \eqref{Zeigs_sde} and denote by
 $z_t^{\gamma} $ the minimal singular value there. Then, again in the Skorohod topology,
$
   \lim_{c \rightarrow \infty} \frac{1}{c} \log z_{c^2 t}^{\gamma/c}  \Rightarrow r_t,
$
where $t \mapsto r_t$ is the diffusion on the positive half-line with generator
$ \frac{1}{2} \frac{d^2}{dr^2} + \gamma \coth (\gamma r) \frac{d}{dr}$.
\end{proposition}

\begin{proof} In order to get a workable formula for the matrix GIG mean, we 
 bring in the more general  $K$-Bessel functions. For $\mathbf{s}=(s_1, \dots, s_r) \in \mathbb{C}^r$ recall the power
function $p_{\mathbf{s}}(X)$ from \eqref{powerfunction} and the invariant measure $\mu_r$ on $\mathcal{P}$ from \eqref{InvMeasure1}, and set
$$
   K_r( \mathbf{s} | A, B) = \frac{1}{2} \int_{\mathcal{P}} p_{\mathbf{s}}(X) e^{- \frac12 \Tr AX - \frac12 \Tr BX^{-1} } d\mu_r(X).
$$
This is actually how Terras introduces the $K$-Bessel functions from the start (see \S 4.2.2 of \cite{Terras}, though keep in mind our inclusion of various factors of $\frac12$ not used there), and reduces to our earlier defined $K_r(s | A, B)$ when $\mathbf{s} = (0,\dots, 0, s)$. 

In terms of the above we have: for any (positive) diagonal matrix $\Lambda = \mathrm{diag}(\lambda_1, \dots, \lambda_r)$,
\begin{equation}
\label{ZZmean}
       [ \kappa_{\mu} (\Lambda, I) ]_{11} = \frac{  K_r ( \mathbf{s} \, | \,  \Lambda, I)}{ K_r (  {\mathbf{s}}^{\prime} \, |  \,   \Lambda, I )},  
\end{equation}
where now $\mathbf{s} = (1, 0, \dots, 0, \mu)$ and  ${\mathbf{s}}^{\prime} = (0,0, \dots, 0, \mu)$. And likewise,
\begin{equation}
\label{ZZoffmean}
   [ \kappa_{\mu}(\Lambda, I) ]_{ii} =  \frac{  K_r ( \mathbf{s} \, | \,  \Lambda_{\sigma_i}, I)}{ K_r (  {\mathbf{s}}^{\prime} \, |  \,   \Lambda, I )},  
\end{equation}
in which $\Lambda_{\sigma_{i}}$ is the matrix arrived at from $\Lambda$ by swapping $\lambda_1$ and $\lambda_i$. 
This uses (again) that  $  \kappa_{\mu} (I,\Lambda) = U^T \kappa_{\mu} (I,U\Lambda U^T) U  $ for orthogonal $U$, here with the choice of $U$ being the corresponding permutation matrix.  Note that 
taking $U$ to be a diagonal orthogonal matrix with $\pm 1$ entries so that $U_{ii}U_{jj}=-1$ for given  $i\neq j$, yields  
$
[ \kappa_{\mu} (I,\Lambda) ]_{i,j}=-[ \kappa_{\mu} (I,\Lambda) ]_{i,j}
$,  so we get that $\kappa_{\mu} (I , \Lambda)$ is actually diagonal.

The ratios  
in \eqref{ZZmean} and \eqref{ZZoffmean} can then be expanded with the help of Terras' induction formula (see Exercise 20 of \cite[\S 4.2.2]{Terras} though note the typo, $\frac{m-n}{2}$ should be $\frac{n-m}{2}$). In the present setting this implies that
\begin{align}
\label{Kinduction}
   K_r( \mathbf{s} | \Lambda, I)  & = \int_{\R^{r-1}} K_1 \left(  \mu - \frac{r-3}{2}  \Bigl| \lambda_1 + \sum_{i=2}^{r} \lambda_i x_i^2 , 1 \right)  \\ 
   & \hspace{1.5cm} \times \, K_{r-1} 
      \left( {\mathbf{s}}^{\prime \prime} \Bigl|  \Lambda^{(1)},  I + xx^T  \right) dx_2 \dots dx_r,  \nonumber
\end{align}
with  $\Lambda^{(1)} = \mbox{diag}(\lambda_2, \dots, \lambda_r)$ and $\mathbf{s}^{\prime \prime}  = $
$(0,\dots, 0, \mu - \frac{1}{2}) \in \R^{(r-1)}$. Applied to  $K_r( \mathbf{s}' | \Lambda, I)$ the result is of course similar, the only difference that the  $K_1$ factor on the right hand side of \eqref{Kinduction} is replaced by 
$K_1 (  \mu - \frac{r-1}{2} \, | \, \lambda_1 + \sum_{i=2}^{r} \lambda_i x_i^2 , 1 )$.

Writing both occurrences of $K_1( \cdot \,  | \, \cdot, \, \cdot)$ in terms of the standard Macdonald function  we have that 
 $ K_1 (  \mu - \frac{r-3}{2}  \, |  \,\psi^2 , 1 )  =  
        \psi^{-{\mu} + \frac{r-3}{2}}  K_{ {\mu}  - \frac{r-3}{2}}  \left(   \psi  \right)  $
and  $K_1 (  \mu - \frac{r-1}{2}\, | \, \psi^2 , 1 ) = \psi^{ - {\mu} + \frac{r-1}{2}} K_{\mu- \frac{r-1}{2}} ( {\psi} ).$
So, with the shorthand,
$$
   \psi = \psi(x,\Lambda) =  \sqrt{\lambda_1 + \sum_{i=2}^{r} \lambda_i x_i^2}, \qquad \mathcal{K}(x, \Lambda^{(1)}) = 
     K_{r-1}  \left( \mathbf{s}^{\prime \prime} \Bigl|   \Lambda^{(1)}, I + x x^T  \right)
$$ 
we record the new expression for \eqref{ZZmean}: making the substitution  $\mu = \frac{r-1}{2} + \gamma$,
\begin{align}
\label{ZZmean1}
       [ \kappa_{\frac{r-1}{2} + \gamma}    (\Lambda, I) ]_{11} 
         =  &  \  \frac{ \int_{\R^{r-1}}  \psi^{- 1-\gamma }  K_{ 1+\gamma}  \left( \psi  \right) 
                                                               \mathcal{K}(x, \Lambda^{(1)}) dx }
           {  \int_{\R^{r-1}}  \psi^{ - \gamma}  K_{ \gamma}  \left( \psi  \right) 
                                                                \mathcal{K}( x, \Lambda^{(1)}) dx},
                                                                        \\                                                               
 %          =  & \,   \frac{1}{\sqrt{\lambda_1}}  \frac{ \int_{\R^{r-1}}  \psi_0^{- 1-\gamma }  K_{ \gamma+1}  \left(  \sqrt{\lambda_1} \psi_0  \right) 
 %                                                              \mathcal{K}( T_{\Lambda} x, \Lambda^{(1)}) dx }
 %          {  \int_{\R^{r-1}}  \psi_0^{ - \gamma}  K_{ \gamma}  \left(  \sqrt{\lambda_{1}} \psi_0  \right) 
  %                                                              \mathcal{K}( T_{\Lambda} x, \Lambda^{(1)}) dx}. \nonumber \\
           :=  &  \,     \frac{1}{\sqrt{\lambda_1}}  
                 \frac{ \int_{\R^{r-1}}  \psi_0^{-1}     K_{ 1+\gamma}  \left(  \sqrt{\lambda_1} \psi_0  \right)  \rho^{(1)}_{\Lambda} (d x)  }
                        { \int_{\R^{r-1}}   K_{ \gamma}  \left(  \sqrt{\lambda_1} \psi_0  \right)  \rho_{\Lambda}^{(1)} (d x) }.     \nonumber                               
\end{align}
In line two we have made the change of variables $T_{\Lambda}: x_i \mapsto  \sqrt{\frac{\lambda_1}{\lambda_i}} x_i$, and have introduced
\begin{equation}
\label{newvariables}
     \psi_0(x) = \sqrt{1  + \sum_{i=2}^{r}  x_i^2}, \qquad  \rho_\Lambda^{(1)}(dx) =
                             \psi_0(x)^{-\gamma} \mathcal{K}( T_{\Lambda} x, \Lambda^{(1)} )  dx.
\end{equation}
By way of  \eqref{ZZoffmean} there are allied expressions  for the other diagonal components of the mean.

Finally returning to \eqref{Zeigs_sde} and setting  $y_i = \log z_i$ that equation becomes
\begin{align}
\label{lastZsde}
   d y_i    =   db_i  \,  + \,  &   \left(  - \gamma    +  e^{-y_i}  \,
                 \frac{ \int_{\R^{r-1}}  \psi_0^{-1}     
                      K_{1+ \gamma}  \left(  e^{-y_i} \, \psi_0  \right)  \rho^{(i)}_{\Lambda} (d x)  }
              { \int_{\R^{r-1}}   K_{ \gamma}  \left( e^{-y_i} , \psi_0  \right)  \rho_{\Lambda}^{(i)} (d x) } \right) dt \\
              & \hspace{1cm}
              + \left( \frac{r-1}{2} + \frac{1}{2} \sum_{j\neq i} \frac{ e^{2 y_i} + e^{2 y_j}}{e^{2 y_i} - e^{2y_j}} \right) dt. 
                \nonumber
\end{align} 
Here $\Lambda$  is now the diagonal matrix  $\Lambda_{ii}  = e^{-2 y_i}$, and we have employed \eqref{ZZmean1} while being a little fluid with notation: $\rho_{\Lambda}^{(i)}$ stands for the the comparable object to $\rho_{\Lambda}^{(1)}$ defined in the same way as in \eqref{newvariables} but for the $i^{\mbox{th}}$ coordinate.

The strategy from this point is: 

\medskip

(i)  Show yet again a separation of scales. That is, for long time it holds that $ y_i \ll y_j $ for all $i < j$ with probability tending to one.  Without the presence of the Macdonald function term in the drift, the same calculation used in Lemma \ref{normlemma} would (yet again) show that the solution  of
\eqref{lastZsde} satisfies  $\frac{1}{t} \log y_i  \rightarrow (r-i)  - \gamma $ for all $i$ with probability one. The claim is that the added drift doesn't affect this appraisal too much
 
\medskip

(ii)  Show that
$$
     \frac{\int_{\R^{r-1}}    \psi_0^{-1} K_{1+ \gamma}  \left(e^{-y_i} \, \psi_0  \right)  \rho^{(1)}_{\Lambda} (d x)  }
              { \int_{\R^{r-1}}   K_{ \gamma}  \left( e^{-y_i} , \psi_0  \right)  \rho_{\Lambda}^{(1)} (d x) }  = \frac{ K_{1+\gamma}(e^{-y_1})}{ K_\gamma(e^{-y_1})} (1+ o(1))
$$
uniformly in $y_1$ as $y_2, \dots y_r  \rightarrow \infty$.     

\medskip

The estimate for (i) follows from known bounds for the $K$-Bessel function at $\infty$. For (ii), using the explicit formulas it is easy to see that the measure $\rho_{\Lambda}^1$ has Gaussian concentration at the point $x =0$ (where one notes  that $\psi_0 = 1$).

Put together, and after the required scaling, the drift in the  equation for $ y =  \frac{1}{c} y_1(c^2t, \gamma/c) $ equals 
$  - \gamma + e^{-c y} \frac{K_{1+\gamma/c} (e^{-cy})}{K_{\gamma/c}(e^{-cy})}$ up to vanishing errors as $c \rightarrow \infty$.  That is, we recover the same calculation needed by Matsumto-Yor in the one dimensional case, and so also the same limit.
 \end{proof}

\end{document}